\documentclass[reqno,11pt]{amsart}
 
\usepackage{amsmath}
\usepackage{amsthm}  
\usepackage{verbatim}
\usepackage{enumerate} 
\usepackage{mathtools}  
\usepackage{amssymb}
\usepackage{mathrsfs}
\usepackage[top=1.5in, bottom=1.5in, left=0.7in, right=0.7in]{geometry}  
\usepackage[colorlinks]{hyperref}
\hypersetup{
	colorlinks,
	citecolor=blue,
	linkcolor=red
}   
\numberwithin{equation}{section}
\usepackage{xypic}
\usepackage{bm}
\usepackage{tikz}
\usepackage{float}

\newtheorem{theorem}{\textbf{Theorem}}[section]
\newtheorem{theorem*}{\textbf{Theorem}}

\newtheorem{proposition}[theorem]{\textbf{Proposition}}
\newtheorem{lemma}[theorem]{\textbf{Lemma}}
\newtheorem{question}[theorem]{\textbf{Question}}

\newtheorem{corollary}[theorem]{\textbf{Corollary}}
\newtheorem{remark}[theorem]{\textbf{Remark}}

\newtheorem{definition/proposition}[theorem]{\textbf{Definition/Proposition}}
\newtheorem{innercustomgeneric}{\customgenericname}
\providecommand{\customgenericname}{}
\newcommand{\newcustomtheorem}[2]{%
	\newenvironment{#1}[1]
	{%
		\renewcommand\customgenericname{#2}%
		\renewcommand\theinnercustomgeneric{##1}%
		\innercustomgeneric
	}
	{\endinnercustomgeneric}
}
\newcustomtheorem{customthm}{Theorem}
\newcustomtheorem{customprop}{Proposition}
\newcustomtheorem{customcoro}{Corollary}

\def\R{\mathbb{R}}
\def\Z{{\mathbb Z}}

\def\C{{\mathbb C}}
\def\D{{\mathbb D}}

\def\Q{{\mathbb Q}}

\def\cA{{\mathcal A}}

\def\cM{{\mathcal M}}

\def\cO{{\mathcal O}}

\def\cS{{\mathcal S}}

\def\rd{{\rm d}}

\def\la{\langle\,}
\def\ra{\,\rangle}

\DeclareMathOperator{\Ima}{im}
\DeclareMathOperator{\ind}{ind}
\DeclareMathOperator{\Hom}{Hom}
\DeclareMathOperator{\Ext}{Ext}
\DeclareMathOperator{\Id}{id}

\DeclareMathOperator{\tor}{Tor}

\DeclareMathOperator{\Crit}{Crit}
\title{On fillings of $\partial(V\times \D)$}
\author{Zhengyi Zhou}

\begin{document}
	\maketitle
\begin{abstract}
We show that any symplectically aspherical/Calabi-Yau filling of $Y:=\partial(V\times \D)$ has vanishing symplectic cohomology for any Liouville domain $V$. In particular, we make no topological requirement on the filling and $c_1(V)$ can be nonzero. Moreover, we show that for any symplectically aspherical/Calabi-Yau filling $W$ of $Y$, the interior $\mathring{W}$ is diffeomorphic to the interior of $V\times \D$ if $\pi_1(Y)$ is abelian and $\dim V\ge 4$. And $W$ is diffeomorphic to $V\times \D$ if moreover the Whitehead group of $\pi_1(Y)$ is trivial. 
\end{abstract}
\section{Introduction}
It was shown by Gromov in his seminal paper \cite{gromov1985pseudo} that symplectically aspherical fillings of $(S^3,\xi_{std})$ are unique symplectically.  Roughly speaking, there exist two orthogonal foliations of holomorphic planes of any symplectically aspherical filling, which recover the diffeomorphism type as well as the symplectic structure. In higher dimensions, Eliashberg-Floer-McDuff \cite{mcduff1990structure} proved that symplectically aspherical fillings of $(S^{2n-1},\xi_{std})$ are diffeomorphic to a ball for $n\ge 3$. The method can be described as considering a moduli space of holomorphic spheres in a partial compactification of the filling, which foliates the filling in a homological sense (some evaluation map has degree $1$).  The homological information turned out to be sufficient to determine the diffeomorphism type by an $h$-cobordism argument. More generally, by a result of Cieliebak \cite{cieliebak2002subcritical}, any subcritical Weinstein domain $W$ splits into $V\times \D$ for a Weinstein domain $V$ and $\D$ is the unit disk in $\C$. Hinted by the natural foliation by the splitting, the ``homological foliation" method was developed by Oancea-Viterbo \cite{oancea2012topology} to show that any symplectically aspherical filling $W$ of a simply connected subcritically fillable contact manifold $Y$ satisfies that $H_*(Y)\to H_*(W)$ is surjective. The homological argument as well as the $h$-cobordism argument were further refined by Barth-Geiges-Zehmisch \cite{barth2016diffeomorphism} to obtain that the diffeomorphism type of symplectically aspherical filling of a simply connected subcritically fillable contact manifold is unique.

In this paper, we study the filling of $\partial(V\times \D)$ for any {\em Liouville} domain $V$ from a Floer theoretic point of view instead of using closed holomorphic curves. The splitting provides with nice Reeb dynamics on the contact boundary. Then the rich algebraic structures on symplectic cohomology allow us to prove the following.
\begin{theorem}\label{thm:main}
	Let $V$ be any Liouville domain, then for any symplectically aspherical/Calabi-Yau (i.e.\ strong filling with $c_1(W)$ torsion) filling $W$ of $Y:=\partial(V\times \D)$, we have the following (all cohomology below is defined over $\Z$).
	\begin{enumerate}
		\item\label{t1} $H^*(W) \to H^*(Y)$ is independent of the filling. In particular, $W$ is necessarily a Liouville filling.
		\item\label{t2} $SH^*(W)=0$ and $SH^*_+(W)$  is independent of the filling.
		\item\label{t3} $W$ is diffeomorphic to $V\times \D$ glued with a homology cobordism from $Y$ to $Y$.
	\end{enumerate}
\end{theorem}
Note that the symplectically aspherical or the Calabi-Yau condition is necessary, since we can always blow up a symplectic filling to change the topology of $W$. The general Liouville case was also discussed in \cite{barth2016diffeomorphism}, where surjectivity on homology was obtained,  which is a corollary of \eqref{t1} above by the universal coefficient theorem. Theorem \ref{thm:main} puts many restrictions on the diffeomorphism type of the filling. Regarding the diffeomorphism type of the filling $W$, one can ask the following three questions.
\begin{enumerate}
	\item\label{1} Is the diffeomorphism type of the open manifold $\mathring{W}$ unique?
	\item\label{2} Is the diffeomorphism type of the manifold with boundary $W$ unique?
	\item\label{3} Is the diffeomorphism type of the manifold with boundary $W$ unique relative to the boundary? i.e.\ is there a diffeomorphism $\phi:W\to V\times \D$ such that $\phi|_{\partial W} = \Id$.
\end{enumerate}
As we shall see, \eqref{1} is related to that whether the homology cobordism in \eqref{t3} of Theorem \ref{thm:main} is an $h$-cobordism. \eqref{2} is related to that whether it is an $s$-cobordism. While \eqref{3} is beyond the reach of the method in this paper. It turns out that we can tackle question \eqref{1} under the assumption that $\pi_1(Y)$ is abelian, as we can study the Floer theory of covering spaces. In particular, we have the following.
\begin{theorem}\label{thm:diff}
	Under the same assumption in Theorem \ref{thm:main}, if in addition $\pi_1(Y)$ is abelian, then $\mathring{W}$ is diffeomorphic to $\mathring{V}\times \mathring{\D}$. If moreover the Whitehead group of $\pi_1(Y)$ is trivial, then $W$ is diffeomorphic to $V\times \D$.
\end{theorem}

Roughly speaking, the proof of Theorem \ref{thm:main} considers the same holomorphic curves in \cite{barth2016diffeomorphism,mcduff1990structure,oancea2012topology}. In some sense the only symplectic information used in \cite{barth2016diffeomorphism,mcduff1990structure,oancea2012topology} was the holomorphic curve with a point constraint, which corresponds to the fact that the evaluation map has degree $1$. In order to get homological information about the filling through duality, one needs to assume that $V$ is Weinstein. In our approach, such curve is again essential as it is responsible for the vanishing of symplectic cohomology. However, we will consider other holomorphic curves with various constraints from $Y$ packaged in Floer theory. Then the ring structure along with the quantitative nature of Floer theory implies Theorem \ref{thm:main}. In addition, the symplectic cohomology framework is flexible enough to work with strong fillings, hence the theorem also applies to Calabi-Yau fillings. By a result of Eliashberg \cite{eliashberg1990filling} and McDuff \cite{mcduff1990structure}, any strong filling of $(S^{3},\xi_{std})$ is a blow-up of the standard ball. Hence Calabi-Yau fillings of $(S^3,\xi_{std})$ are also unique. Our result can be viewed as a generalization of that.

Theorem \ref{thm:main} also has certain overlap to the results in \cite{zhou2019symplectic}, i.e. those $V$ with vanishing first Chern class, in particular, boundary of subcritically fillable contact manifolds with vanishing first Chern class. The main difference is the following, \cite{zhou2019symplectic} uses the fact those contact manifolds has only the trivial $\Z$-graded augmentation by degree reasons, hence can be applied to a very different class of contact manifolds called asymptotically dynamically convex manifolds, which includes boundaries of flexible Weinstein domains with vanishing first Chern class \cite{lazarev2016contact} and links of isolated terminal singularities \cite{mclean2016reeb}. Moreover, similar phenomena extend to many other structures on symplectic cohomology \cite{zhou2019symplectic2}. However, the index consideration brings a major drawback that we need to require the filling to have mild topology properties ($c_1=0$ and $\pi_1$ injective) in order to obtain a $\Z$-grading. On the other hand, we will not need any grading requirement for Theorem \ref{thm:main}. In fact, when we drop the grading requirement, there are always different augmentations (coming from blow-ups). Unlike the strategy of exploiting the uniqueness of augmentations in \cite{zhou2019symplectic2,zhou2019symplectic}, we will make use of the nice Reeb dynamics induced from the splitting setup and explore structures that are independent of augmentations. Such changes of perspectives result in many differences compared to \cite{zhou2019symplectic}. In addition to dropping the topological assumptions in \cite{zhou2019symplectic}, the invariance of $SH^*_+(W)$ is not a priori fact but rather a posteriori consequence of $SH^*(W)=0$. The upshot is that the topology is seen by all the Reeb orbits wrapping around $V$ once and a certain map from the (filtered) positive symplectic cohomology does not depend on augmentations. 

In fact, our proof shows that $1$ is in the image of $SH_+^*(W)\to H^{*+1}(Y)\to H^0(Y)$ for any strong filling $W$, assuming symplectic cohomology and its positive version is well-defined for general strong fillings. Such phenomena, studied in \cite{zhou2019symplectic2,zhou2019symplectic}, has gone beyond the situation of having only the trivial ($\Z$-graded) augmentation. Then the symplectically asphericality is used to show $1+A$ is a unit in the quantum cohomology $QH^*(W;\Lambda)$ for $A\in\oplus_{i\ge 1} H^{2i}(W;\Lambda)$, which is crucial for the vanishing of symplectic cohomology. As by \cite{ritter2014floer}, when the filling is not symplectically aspherical, we do have zero divisors in the form of $1+A$ even for the standard contact sphere $(S^{2n-1},\xi_{\mathrm{std}})$. On the other hand, it seems that $SH_+^*(W)\to H^{*+1}(Y)$ is also independent of strong fillings, at least it holds for the standard ball and its blow-up $\cO(-1)$ as fillings of $(S^{2n-1},\xi_{\mathrm{std}})$. In the Calabi-Yau case, $A$ is necesarrily zero by degree reasons and $1$ is always a unit in $QH^*(W;\Lambda)$. In general, Theorem \ref{thm:main} holds as long as we know that there is no zero divisor of $QH^*(W;\Lambda)$ in the form of $1+A$ for $A\in \oplus_{i>0}H^{2i}(W;\Lambda)$, where $\Lambda$ is the Novikov field. In particular, we have the following.

\begin{corollary}\label{cor:strong}
	Let $W$ be a (semi-positive) strong filling of $Y:=\partial (V\times \D)$ for $\dim V=2n$. If there is no embedded symplectic sphere $S$ with $2-n\le c_1(S)\le 2n-1$, then $W$ is a Liouville filling. 
\end{corollary}
The semi-positive assumption is only for the definition of symplectic cohomology without using any virtual technique, and should be irrelevant once one constructs (positive) symplectic cohomology for general strong fillings. Note that when $n=1$, the assumption is equivalent to that $W$ is minimal, i.e.\ there is no exceptional sphere in $W$. In such case, Corollary \ref{cor:strong} is also implied by a result of Wendl \cite{wendl2010strongly}, since $V\times \D$ is subcritical and $\partial (V\times \D)$ is supported by a planar open book. In dimensions higher than $4$, there are many other operations to modify a filling other than blowing up a point. Corollary \ref{cor:strong} implies that any birational surgery that we can apply on $W$ to destroy symplectically asphericality will create symplectic spheres with small first Chern class.

\subsection*{Acknowledgement} The author is supported by the National Science Foundation Grant No. DMS1638352. It is a great pleasure to acknowledge the Institute for Advanced Study for its warm hospitality. The author would like to thank the referee for suggestions that improve the exposition.

\section{Symplectic cohomology}
\subsection{Contact forms on $Y$}\label{ss:contact}
Let $\lambda$ denote the Liouville form on the Liouville domain $V$. In the following, we describe a special contact form on $Y:=\partial (V\times \D)$, which allows us to single out the Reeb orbits corresponding to critical points for a Morse function on $V$.  Let $r$ denote the collar coordinate on $\partial V$, such that the completed Liouville manifold $(\widehat{V},\widehat{\lambda})$ is given by $V\cup \partial V \times (1,\infty)_r$ with $\widehat{\lambda}=\lambda$ on $V$ and $\widehat{\lambda}=r(\lambda|_{\partial V})$ on $\partial V \times (1,\infty)_r$. Note that by following the negative flow of the Liouville vector field $r\partial_r$, the $r$ coordinate continues to exist in the interior of $V$ for $r\in (0,1)$. We first fix a Morse function $f$ on $V$, such that the following holds.
\begin{enumerate}
	\item $\min f =0$ and $\max f=1$ which is attained at $\partial V$. 
	\item For $r\in (\frac{1}{2}, 1)$, $f$ only depends on $r$ and 	$1\ge \partial_r f > 0$ and $\partial_r f|_{r=1}=1$.
	\item $f$ is self-indexing in the sense that $f(p)=\frac{1}{2n+1-\ind p}$ for a critical point $p$ with Morse index $\ind p>0$ and $f$ has a unique local minimum $0$ (hence the global minimum). We may assume $\ind p \le 2n-1$, since $V$ is an open manifold.
\end{enumerate}
Then we can find a smooth family of decreasing functions $g_{\epsilon}:[0,\frac{1+\epsilon}{1+\epsilon f(\frac{1}{2})}]\to [\frac{1}{2},\frac{3}{4}]$ such that the following holds.
\begin{enumerate}
	\item $g_{\epsilon}(x) = f^{-1}(\frac{1+\epsilon}{\epsilon x}-\frac{1}{\epsilon})$ for $x$ near $\frac{1+\epsilon}{1+\epsilon f(\frac{1}{2})}$, hence $g_{\epsilon}(\frac{1+\epsilon}{1+\epsilon f(\frac{1}{2})})=\frac{1}{2}$.
	\item $g_{\epsilon}(x)=\frac{3}{4}$ when $x \le \frac{1}{4}$.
\end{enumerate}

Let $\rho$ denote the radical coordinate in $\C$ and $\D(r)$ denote the disk of radius $r$. In this paper, we fix $\frac{1}{2\pi}\rho^2 \rd \theta$ as the Liouville form on $\C$. With the data above, for $\epsilon>0$, we have a contact type hypersurface $Y_\epsilon$ in the completed Liouville manifold $(\widehat{V}\times \C, \widehat{\lambda}\oplus \frac{1}{2\pi}\rho^2 \rd \theta)$ given as follows:
\begin{enumerate}
	\item in the region $V\times \C$, $Y_{\epsilon}$ is given by $\rho^2 = \frac{1+\epsilon}{1+\epsilon f}$,
	\item in the region $\widehat{V} \times \D(\sqrt{\frac{1+\epsilon}{1+\epsilon f(\frac{1}{2})}})$, $Y_{\epsilon}$ is given by $r=g_{\epsilon}(\rho^2)$. 
\end{enumerate}
Then our conditions on $g_{\epsilon}$ guarantee that $Y_{\epsilon}$ closes up to a smooth closed hypersurface, which can be described pictorially as below. In particular, $Y_{\epsilon}$ is a smooth family of hypersurfaces for $\epsilon \ge 0$.
\begin{figure}[H]
\begin{center}
	\begin{tikzpicture}[scale=1.5]
	\draw (0,0) to (4,0) to (4,2) to (0,2) to (0,0); 
	\draw[dashed] (0,0) to [out=320, in=180] (2,-0.5) to [out=0, in=220] (4,0);
	\draw[dashed] (0,2) to [out=40, in=180] (2,2.5) to [out=0, in=140] (4,2);
	\draw (0.5,0) to [out=270, in =170] (1,-0.45) to [out=-10, in =180 ](2,-0.5);
	\draw (0.5,2) to [out=90, in =190] (1,2.45)  to [out=10, in =180 ] (2, 2.5);
	\draw (3.5,0) to [out=270, in =10] (3,-0.45)  to [out= 190, in = 0 ] (2, -0.5);
	\draw (3.5,2) to [out=90, in =-10] (3,2.45)  to [out= 170, in = 0 ] (2, 2.5);
	\draw (0.5,0) to (0.5,2);
	\draw (3.5, 0) to (3.5, 2);
	\node at (2,1) {$V\times \D$};
	\node at (2,2.7) {$\rho^2 = \frac{1+\epsilon}{1+\epsilon f}$};
	\node at (0.35, 1) {$Y_{\epsilon}$};
	
	\draw (6,0) to (10,0) to (10,2) to (6,2) to (6,0); 
	\draw (7,0) to [out=180, in=270] (6.5,0.5) to (6.5,1.5) to [out=90, in=180] (7, 2);
	\draw (9,0) to [out=0, in=270] (9.5,0.5) to (9.5,1.5) to [out=90, in=0] (9, 2);  
	\node at (8,1) {$V\times \D$};
	\node at (6.35,1) {$Y_0$};
	
	\draw[->] (4.5, 1) to (5.5,1);
	\node at (5,1.2) {$\epsilon \to 0$};
	\end{tikzpicture}
	\caption{The contact hypersurface $Y_{\epsilon}$}
\end{center}
\end{figure}
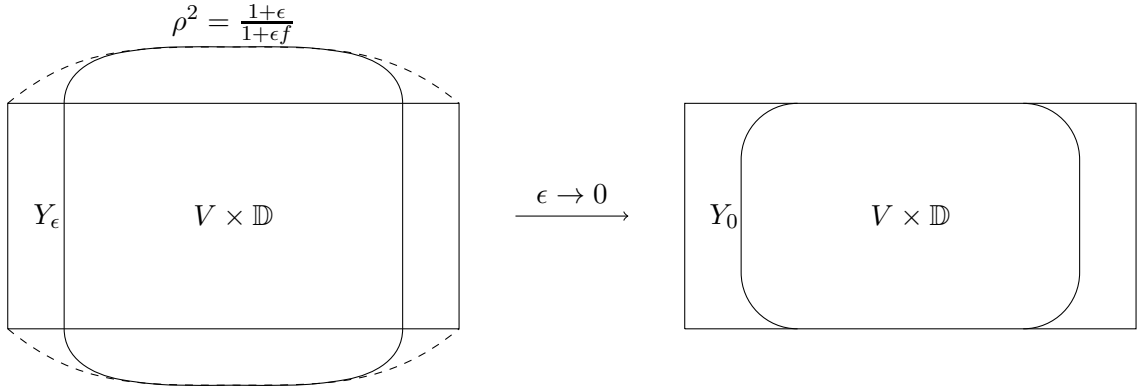

\begin{proposition}
	For $\epsilon$ sufficiently small, $Y_{\epsilon}$ is a contact type hypersurface, i.e.\ the restriction of the Liouville form $\widehat{\lambda}\oplus \frac{1}{2\pi}\rho^2\rd \theta$ gives a contact form on $Y_{\epsilon}\simeq Y$.
	\end{proposition}
\begin{proof}
	In view of \cite[Remark 6.5]{zhou2019symplectic}, it is sufficient to prove that 
	$\frac{1+\epsilon}{1+\epsilon f}-X_{\lambda}(\frac{1+\epsilon}{1+\epsilon f})>0$ on $V$ for the Liouville vector $X_{\lambda}$ and $g_{\epsilon}(\rho^2)-\frac{1}{2}\rho \partial_{\rho}g_{\epsilon}(\rho^2)=g_{\epsilon}(x)-\partial_x g_{\epsilon}(x)>0$ for $x=\rho^2$ on $\D(\sqrt{\frac{1+\epsilon}{1+\epsilon f(\frac{1}{2})}})$. Note that 
	$$\frac{1+\epsilon}{1+\epsilon f}-X_{\lambda}(\frac{1+\epsilon}{1+\epsilon f}) = \frac{1+\epsilon}{1+\epsilon f}+\frac{(1+\epsilon)\epsilon X_{\lambda}(f)}{(1+\epsilon f)^2}=\frac{1+\epsilon}{(1+\epsilon f)^2}(1+\epsilon f+\epsilon X_{\lambda}(f)),$$
	which is positive when $\epsilon$ is small enough. Since $g_{\epsilon}(x)>0$ and $\partial_x g_{\epsilon}(x)<0$, we also have $g_{\epsilon}(x)-\partial_x g_{\epsilon}(x)>0$.
\end{proof}
In the following, we assume that the Reeb dynamics on $\partial V$ are non-degenerate and the shortest Reeb orbit on $\partial V$ has period at least $5$, this can always be achieved by perturbing and scaling the Liouville form. Then we have the following.
\begin{proposition}
	For $\epsilon$ small enough and positive, any Reeb orbit on $Y_\epsilon$ with period $<2$ must be the circle $\gamma_p$ over a critical point $p$ of $f$ and is non-degenerate. Moreover, the period is given by $\frac{1+\epsilon}{1+\epsilon f(p)}$. When $\epsilon=0$, any Reeb orbit on $Y_0$ with period $<2$ is a circle over a point in $V\backslash\left(\partial V \times (\frac{1}{2},1]\right)$.
\end{proposition}
\begin{proof}
	It follows from the same argument as in \cite[Proposition 6.7]{zhou2019symplectic}, any Reeb orbit touches the region $\left(\widehat{V}\times \D(\sqrt{\frac{1+\epsilon}{1+\epsilon f(\frac{1}{2})}})\right) \cap Y_{\epsilon}$ must be the in form of $(\gamma(At), \rho e ^{iBt+\theta_0})$ for Reeb orbit $\gamma $ on $\partial V$, where 
	$$A=\frac{1}{g_{\epsilon}(\rho^2)-\rho^2g'_{\epsilon}(\rho^2)} \le  2.$$
	In particular, the period of such orbits must be greater than $2$. For the rest of $Y_{\epsilon}$, following the same argument as in \cite[Theorem 6.3]{zhou2019symplectic}, for $\epsilon$ small enough, all period orbits of period smaller than $2$ on $\left(V\times \C \right)\cap Y_{\epsilon}$ must be the simple circle over some critical point $p$ of $f$ with the prescribed period. When $\epsilon=0$, the situation on $Y_0\cap \left(\widehat{V}\times \mathring{\D}\right)$ is the same as before that all Reeb orbits have period greater than $2$.  On the remaining portion, the Reeb vector field is $2\pi \partial_\theta$, where $\theta$ is the angular coordinate on $\C$. Hence the claim follows.
\end{proof}

\begin{remark}
    A more general discussion of the contact properties (Reeb flow, Conley-Zehnder indices) of such hypersurfaces can be found in \cite[\S 6]{zhou2019symplectic}. The specific properties, e.g.\ the self-indexing property of $f$, are for the computations in Proposition \ref{prop:nice}.
\end{remark}

\subsection{Symplectic cohomology}\label{ss:sympcoh}
In the following, we recall models of symplectic cohomology that will be used in this paper. We will follow the autonomous setting in \cite{bourgeois2009symplectic}, see also \cite[\S 2]{zhou2020mathbb} for the exactly same setup and \cite[Remark 2.5]{zhou2020mathbb} for comparisons and relations with other models. Let $W$ be a(n exact) filling of $Y$ and $\cS(\partial W)$ be the set of periods of the Reeb flow on $Y$, which is a discrete set, when the Reeb dynamics are non-degenerate, which we shall assume.  We will consider admissible Hamiltonians  of one of the following two forms.
\begin{enumerate}[(I)]	
	\item\label{I} $H=0$ on $W$, $H=h(r)$ on $\partial W \times (1,\infty)$ and $h'(r)=a \notin \cS(\partial W)$ for $r\ge 1+w$ and $h''(r)>0$ for $1<r<1+w$, here $w$ is called the width of the Hamiltonian.
	\item\label{II} Or $H\le 0$ and is $C^2$ small on $W$ and has the same property on $\partial W \times (1,\infty)$ as in \eqref{I}, such that all periodic orbits of $X_H$ are either non-degenerate critical points of $H$ on $W$ or non-constant orbits in $\partial W \times (1,\infty)$.
\end{enumerate}
In particular, any non-constant orbit of $X_H$ corresponds to some Reeb orbit on $\partial W$ shifted in the $r$-direction. Our symplectic action uese the cohomological convention
\begin{equation}\label{eqn:action}
\cA(\gamma) = -\int_{\gamma} \widehat{\lambda} + \int_{\gamma} H,
\end{equation}
Our convention for $X_H$ is $\widehat{\omega} (\cdot, X_H)=\rd H$.  Here $\widehat{\lambda},\widehat{\omega}$ are the completed Liouville, symplectic forms of the completed Liouville manifold $\widehat{W}$. We choose a time dependent $\widehat{\omega}$ compatible almost complex structure $J$, such that the restriction on $W$ is time independent. Moreover, $J$ is cylindrical convex\footnote{i.e.\ $J$ is $\rd \widehat{\lambda}$ compatible and $\widehat{\lambda} \circ J = \rd r$.} near every $r$, such that $h'(r)$ is the period of a Reeb orbit, which guarantees the validity of the integrated maximum principle \cite{cieliebak2018symplectic}.   Then we pick two different generic points $\hat{\gamma}$ and $\check{\gamma}$ on $\Ima \overline{\gamma}$,  where $\overline{\gamma}$ is the $S^1$-family of the non-constant orbits of $X_H$ corresponding to the Reeb orbit $\gamma$. This is equivalent to choosing a Morse function $g_{\overline{\gamma}}$ with one maximum and one minimum on $\Ima \overline{\gamma}$ in \cite[\S 3]{bourgeois2009symplectic}. By \cite[Lemma 3.4]{bourgeois2009symplectic}, the Morse function $g_{\overline{\gamma}}$ can be used to perturb the Hamiltonian $H$ to get two non-degenerate orbits from $\overline{\gamma}$, which are often denoted by $\hat{\gamma}$ and $\check{\gamma}$ in literatures with $\mu_{CZ}(\hat{\gamma})=\mu_{CZ}(\gamma)+1$ and $\mu_{CZ}(\check{\gamma})=\mu_{CZ}(\gamma)$, where the former $\mu_{CZ}$ is the Conley-Zehnder index of Hamiltonian orbits $\hat{\gamma}$ and $\check{\gamma}$ and the latter $\mu_{CZ}$ is the Conley-Zehnder index of the Reeb orbit $\gamma$. In the case of \eqref{II}, the cochain complex is the free $\Z$-module generated by $\Crit(H)$ and $\check{\gamma}_p, \hat{\gamma}_p$ for $p\in \Crit(f)$. The differential is computed by counting rigid cascades, which can be described pictorially as follows,
\begin{figure}[H]
	\begin{tikzpicture}
	\draw (0,0) to [out=90, in = 180] (0.5, 0.25) to [out=0, in=90] (1,0) to [out=270, in=0] (0.5,-0.25)
	to [out = 180, in=270] (0,0) to (0,-1);
	\draw[dotted] (0,-1) to  [out=90, in = 180] (0.5, -0.75) to [out=0, in=90] (1,-1);
	\draw (1,-1) to [out=270, in=0] (0.5,-1.25) to [out = 180, in=270] (0,-1);
	\draw (1,0) to (1,-1);
	\draw[->] (1,-1) to (1.25,-1);
	\draw (1.25,-1) to (1.5,-1);
	\draw (1.5,-1) to [out=90, in = 180] (2, -0.75) to [out=0, in=90] (2.5,-1) to [out=270, in=0] (2,-1.25)
	to [out = 180, in=270] (1.5,-1) to (1.5,-2);
	\draw[dotted] (1.5,-2) to  [out=90, in = 180] (2, -1.75) to [out=0, in=90] (2.5,-2);
	\draw (2.5,-2) to [out=270, in=0] (2,-2.25) to [out = 180, in=270] (1.5,-2);
	\draw (2.5,-1) to (2.5,-2);
	\draw[->] (2.5,-2) to (2.75,-2);
	\draw (2.75,-2) to (3,-2);
	\draw[->] (-0.5,0) to (-0.25,0);
	\draw (-0.25,0) to (0,0);
	\node at (0.5,-0.5) {$u_1$};
	\node at (2, -1.5) {$u_2$};
	\end{tikzpicture}
	\begin{tikzpicture}
	\draw (0,0) to [out=90, in = 180] (0.5, 0.25) to [out=0, in=90] (1,0) to [out=270, in=0] (0.5,-0.25)
	to [out = 180, in=270] (0,0) to (0,-1);
	\draw[dotted] (0,-1) to  [out=90, in = 180] (0.5, -0.75) to [out=0, in=90] (1,-1);
	\draw (1,-1) to [out=270, in=0] (0.5,-1.25) to [out = 180, in=270] (0,-1);
	\draw (1,0) to (1,-1);
	\draw[->] (1,-1) to (1.25,-1);
	\draw (1.25,-1) to (1.5,-1);
	\draw (1.5,-1) to [out=90, in = 180] (2, -0.75) to [out=0, in=90] (2.5,-1) to [out=270, in=0] (2,-1.25)
	to [out = 180, in=270] (1.5,-1) to (1.5,-2);
	\draw (2.5,-2) to [out=270, in=0] (2,-2.5) to [out = 180, in=270] (1.5,-2);
	\draw (2.5,-1) to (2.5,-2);
	\draw[->] (-0.5,0) to (-0.25,0);
	\draw (-0.25,0) to (0,0);
	\node at (2,-2.5) [circle,fill,inner sep=1.5pt]{};
	\node at (2.3, -3) {$p\in \Crit(H)$};  
	\node at (0.5,-0.5) {$u_1$};
	\node at (2, -1.5) {$u_2$};
	\end{tikzpicture}
	\caption{$d_+$ and $d_{+,0}$ from $2$ level cascades in case \eqref{II}}
\end{figure}
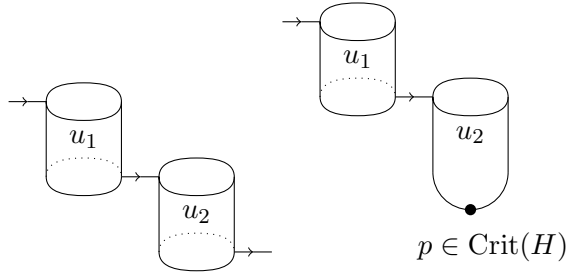
\begin{enumerate}
	\item The horizontal arrow is flowing in $\Ima \overline{\gamma}$ towards $\check{\gamma}$.
	\item $u$ is a solution to the Floer equation $\partial_s u+J_t(\partial_tu-X_H)=0$ modulo the $\R$ translation.
	\item Every intersection point of the line with the surface satisfies the obvious matching condition.
\end{enumerate}

In the case of \eqref{I}, the constant periodic orbits of $X_H$ are parameterized by $W$, which are Morse-Bott non-degenerate except for points on $\partial W$\footnote{How such failure of Morse-Bott non-degeneracy will not be seen by the moduli spaces for symplectic cohomology as explained in \cite[Proposition 2.6]{zhou2019symplectic}.}. To break the Morse-Bott symmetry, we fix a metric and a Morse function $h$ on $W$ such that $\partial_r h>0$ on $\partial W$ . Then the cochain complex is the free $\Z$-module generated by $p\in \Crit(h)$ and $\check{\gamma}_p,\hat{\gamma}_p$. 
The differential is computed by counting rigid cascades, which can be described pictorially as follows, with the only difference of extra gradient flow of $h$.
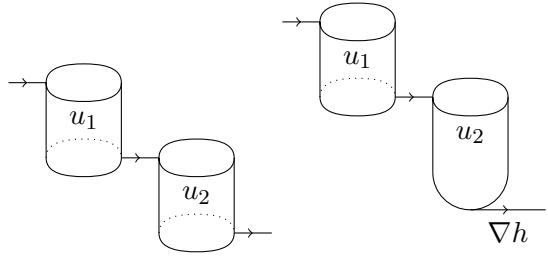
\begin{figure}[H]
	\begin{tikzpicture}
	\draw (0,0) to [out=90, in = 180] (0.5, 0.25) to [out=0, in=90] (1,0) to [out=270, in=0] (0.5,-0.25)
	to [out = 180, in=270] (0,0) to (0,-1);
	\draw[dotted] (0,-1) to  [out=90, in = 180] (0.5, -0.75) to [out=0, in=90] (1,-1);
	\draw (1,-1) to [out=270, in=0] (0.5,-1.25) to [out = 180, in=270] (0,-1);
	\draw (1,0) to (1,-1);
	\draw[->] (1,-1) to (1.25,-1);
	\draw (1.25,-1) to (1.5,-1);
	\draw (1.5,-1) to [out=90, in = 180] (2, -0.75) to [out=0, in=90] (2.5,-1) to [out=270, in=0] (2,-1.25)
	to [out = 180, in=270] (1.5,-1) to (1.5,-2);
	\draw[dotted] (1.5,-2) to  [out=90, in = 180] (2, -1.75) to [out=0, in=90] (2.5,-2);
	\draw (2.5,-2) to [out=270, in=0] (2,-2.25) to [out = 180, in=270] (1.5,-2);
	\draw (2.5,-1) to (2.5,-2);
	\draw[->] (2.5,-2) to (2.75,-2);
	\draw (2.75,-2) to (3,-2);
	\draw[->] (-0.5,0) to (-0.25,0);
	\draw (-0.25,0) to (0,0);
	\node at (0.5,-0.5) {$u_1$};
	\node at (2, -1.5) {$u_2$};
	\end{tikzpicture}
	\begin{tikzpicture}
	\draw (0,0) to [out=90, in = 180] (0.5, 0.25) to [out=0, in=90] (1,0) to [out=270, in=0] (0.5,-0.25)
	to [out = 180, in=270] (0,0) to (0,-1);
	\draw[dotted] (0,-1) to  [out=90, in = 180] (0.5, -0.75) to [out=0, in=90] (1,-1);
	\draw (1,-1) to [out=270, in=0] (0.5,-1.25) to [out = 180, in=270] (0,-1);
	\draw (1,0) to (1,-1);
	\draw[->] (1,-1) to (1.25,-1);
	\draw (1.25,-1) to (1.5,-1);
	\draw (1.5,-1) to [out=90, in = 180] (2, -0.75) to [out=0, in=90] (2.5,-1) to [out=270, in=0] (2,-1.25)
	to [out = 180, in=270] (1.5,-1) to (1.5,-2);
	\draw (2.5,-2) to [out=270, in=0] (2,-2.5) to [out = 180, in=270] (1.5,-2);
	\draw (2.5,-1) to (2.5,-2);
	\draw[->] (-0.5,0) to (-0.25,0);
	\draw (-0.25,0) to (0,0);
	\draw[->] (2,-2.5) to (2.5,-2.5);
	\draw (2.5,-2.5) to (3,-2.5);
	\node at (2.5, -2.8) {$\nabla h$};  
	\node at (0.5,-0.5) {$u_1$};
	\node at (2, -1.5) {$u_2$};
	\end{tikzpicture}
	\caption{$d_+$ and $d_{+,0}$ from $2$ level cascades in case \eqref{I}}
\end{figure}
More formally, the differential is defined by counting the following compactified moduli spaces.
\begin{enumerate}
	\item For $p,q\in \Crit(h)$, 
	$$\cM_{p,q}:=\overline{\{ \gamma:\R_s \to W|\gamma'+\nabla h =0, \lim_{s\to \infty} \gamma=p,\lim_{s\to -\infty} \gamma=q \}/\R}.$$
	\item For $\gamma_+,\gamma_- \in \{\check{\gamma},\hat{\gamma}| \forall S^1 \text{ family of orbits } \overline{\gamma}\}$, a $k$-cascade from $\gamma_+$ to $\gamma_-$ is a tuple $(u_1,l_1,\ldots,l_{k-1},u_k)$, such that
	\begin{enumerate}
		\item $l_i$ are positive real numbers.
		\item nontrivial $u_i\in \{u:\R_s\times S^1_t\to \widehat{W}| \partial_su+J_t(\partial_tu-X_H)=0, \lim_{s\to \infty} u \in \overline{\gamma}_{i-1}, \lim_{s\to -\infty} u \in \overline{\gamma}_{i}\}/\R$ such that $\gamma_+\in \overline{\gamma}_0$ and $\gamma_-\in \overline{\gamma}_{k}$, where the $\R$ action is the translation on $s$.
		\item $\phi_{-\nabla g_{\overline{\gamma}_i}}^{l_i}(\lim_{s\to -\infty}u_i(s,0))=\lim_{s\to \infty}u_{i+1}(s,0)$ for $1\le i\le k-1$, $\gamma_+=\lim_{t\to \infty}\phi^{-t}_{-\nabla g_{\overline{\gamma}_0}}(\lim_{s\to \infty} u_1(s,0))$, and $\gamma_-=\lim_{t\to \infty}\phi^{t}_{-\nabla g_{\overline{\gamma}_k}}(\lim_{s\to -\infty} u_k(s,0))$, where $\phi^t_{-\nabla g_{\overline{\gamma}}}$ is the time $t$ flow of $-\nabla g_{\overline{\gamma}}$ on $\Ima \overline{\gamma}$.
	\end{enumerate}
	Then we define $\cM_{\gamma_+,\gamma_-}$ to be the compactification of the space of all cascades from $\gamma_+$ to $\gamma_-$. The compactification involves the usual Hamiltonian-Floer breaking of $u_i$ as well as degeneration corresponding to $l_i=0,\infty$.  The $l_i=0$ degeneration is equivalent to a Hamiltonian-Floer breaking $\lim_{s\to -\infty}u_i=\lim_{s\to \infty}u_{i+1}$. In particular, they can be glued or paired, hence do not contribute (algebraically) to the boundary of $\cM_{\gamma_+,\gamma_-}$. The $l_i=\infty$ degeneration is equivalent to a Morse breaking for $g_{\overline{\gamma}_i}$, which will contribute to the boundary of $\cM_{\gamma_+,\gamma_-}$.

	\item For $\gamma_+ \in  \{\check{\gamma},\hat{\gamma}| \forall S^1 \text{ family of orbits } \overline{\gamma}\}$ and $q\in \Crit(h)$, a $k$-cascades from $\gamma_+$ to $q$ is a tuple  $(l_0,u_1,l_1,\ldots,u_k,l_{k})$ as before, except
	\begin{enumerate}
		\item $u_k\in \{u:\C \to \widehat{W}|\partial_s u+J_t(\partial_tu-X_H)=0,\lim_{s\to \infty} u \in \overline{\gamma}_{k-1}, u(0)\in \mathring{W}\}/\R$, where we use the identification $\R\times S^1 \to \C^*, (s,t)\mapsto e^{2\pi(s+it)}$. $\mathring{W}$ is the interior of $W$, where the Floer equation is $\overline{\partial}_J u=0$, hence the removal of singularity implies that $u(0)$ is a well-defined notation.
		\item $q=\lim_{t\to \infty}\phi^t_{\nabla h}(u_k(0))$.
	\end{enumerate}
    Then $\cM_{\gamma_+,q}$ is defined to be the compactification of the space of all cascades from $\gamma_+$ to $q$. 
\end{enumerate}
All of the moduli spaces above can be equipped with coherent orientations.

The type \eqref{II} Hamiltonians were used in \cite{bourgeois2009symplectic}, while the a type \eqref{I} Hamiltonian is a hybrid of \cite{bourgeois2009symplectic,zhou2019symplectic} and was used in \cite{zhou2020mathbb}. We use $(C(H),d)$ to denote the total cochain complex in both cases. The cochain complex generated by $p\in \Crit(h)$ or $p\in \Crit(H|_W)$ is a subcomplex $(C_0,d_0)$, which computes the cohomology of $W$. The cochain complex generated by $\check{\gamma},\hat{\gamma}$ for all Reeb orbits $\gamma$ with period smaller than $a$ is quotient complex $(C_+(H),d_+)$. The connecting map $C_+(H)\to C_0$ is denoted by $d_{+,0}$.  In view of the notation above, the differentials are defined as follows.
\begin{eqnarray*}
d_0 (p) &= &\sum_{q,\dim \cM_{p,q}=0} \# \cM_{p,q} q,\\
d_+ (\gamma_+) &= &\sum_{\gamma_-,\dim \cM_{\gamma_+,\gamma_-}=0} \# \cM_{\gamma_+,\gamma_-} \gamma_-,\\
d_{+,0} (\gamma_+) &= &\sum_{q,\dim \cM_{\gamma_+,q}=0} \# \cM_{\gamma_+,q} q,
\end{eqnarray*}
where $\# \cM$ denotes the signed count of the zero-dimensional compact moduli space $\cM$.

Since $a\notin \cS(\partial W)$, a continuation map argument (\cite[Proposition 2.8]{zhou2020mathbb}) implies that the cohomology is independent of the choice of $H$, and the associated cohomology is called the filtered (positive) symplectic cohomology $SH^{\le a}(W)$(and $SH^{\le a}_+(W)$), which can also be defined using the action filtration as in \cite{oancea2006kunneth}. Moreover, we have the following tautological long exact sequence (with $\Z/2$ grading in general by $n-\mu_{CZ}$),
$$\ldots \to H^*(W) \to SH^{*,\le a}(W) \to SH_+^{*,\le a}(W) \to H^{*+1}(W) \to \ldots.$$
For $a<b\notin \cS(Y)$, there is a continuation map $\iota_{a,b}: SH^{*,\le a}(W)\to SH^{*,\le b}(W)$ as well as on the positive symplectic cohomology, which are isomorphisms given that $[a,b]\cap \cS(\partial W)=\emptyset$. The continuation maps are compatible with the long exact sequence above, and the direct limit of the filtered (positive) symplectic cohomology is the (positive) symplectic cohomology $SH^*(W)$(and $SH^*_+(W)$).

\begin{remark}
    Both type \eqref{I} and \eqref{II} Hamiltonians give rise to isomorphic (positive) symplectic cohomology by a continuation map argument \cite[Proposition 2.10]{zhou2019symplectic}. The type \eqref{I} Hamilotnian is better suited for neck-stretching as it is zero along any contact hypersurface in $W$ but a type \eqref{II} Hamiltonian provides an easier setup for the K\"unneth formula below.
\end{remark}

\subsection{The K\"unneth formula}
It was shown by Oancea \cite{oancea2006kunneth} that the K\"unneth formula holds for symplectic cohomology. In particular, we have that $SH^*(V\times \D) = 0$, which is crucial for this paper.  However, in order to obtain the K\"unneth formula for $V\times W$, one does not use those Hamiltonians in \S \ref{ss:sympcoh}. Instead, one uses Hamiltonians in the form of $H\oplus K:=\pi_1^*H+\pi_2^*K$, where $H,K$ are admissible type \eqref{II} Hamiltonians on $V,W$ respectively and $\pi_1,\pi_2$ are two natural projections on $V\times W$. Note that $H\oplus K$ is {\em not} admissible on $V\times W$. When using a splitting almost complex structure $J_1\oplus J_2$ for admissible almost complex structures $J_1, J_2$ on $V,W$ respectively,  we have that $C(H\oplus K)=C(H)\otimes C(K)$. Note that $J_1\oplus J_2$ is also {\em not} admissible. The key step in the proof of the K\"unneth formula is relating the splitting model with the admissible model by a continuation map. Computation is much easier in the splitting model, in particular, we will compute the standard case $V\times \D$ using such splitting model. For this purpose, we first introduce some notations and properties that will be used in this paper.

Let $V,W$ be two Liouville domains with the induced contact forms on the boundary non-degenerate. We fix $a>0 \notin \cS(\partial V), b>0 \notin \cS(\partial W)$ and two admissible Hamiltonians $H,K$ with slope $a$ and $b$ of type \eqref{II} on $V,W$ respectively. We fix generic admissible almost complex structures $J_1,J_2$ on $V, W$ respectively. Then the Hamiltonian-Floer cochain complex $C(H\oplus K)$ of $H \oplus K$ using $J_1\oplus J_2$ is the tensor $C(H)\otimes C(K)$.\footnote{The periodic orbits of $X_{H+K}$ are isolated, in $S^1$ family or $S^1\times S^1$ family. Since the degenerate orbits are Morse-Bott non-degenerate, the cascades construction in \cite{bourgeois2009symplectic} can be adapted to such case.} The subcomplex $C_0(H)\otimes C_0(K)$ is a Morse complex on $V\times W$ and the corresponding quotient complex denoted by $C_+(H\oplus K)$ has a decomposition as $\left(C_0(H)\otimes C_+(K)\right)\oplus \left(C_+(H)\otimes C_+(K)\right) \oplus \left(C_+(H)\otimes C_0(K)\right)$. The cohomology of $C(H\oplus K),C_+(H\oplus K)$ do not depend on the choice of $H,K,J_1,J_2$ as before (\cite[Proposition 2.8]{zhou2020mathbb}), and will be denoted by $SH^{*, \le a, \le b}(V\times W)$ and $SH^{*,\le a, \le b}_+(V\times W)$ respectively. The product version of the continuation map $\iota_{a_1,a_2}:SH^{*,\le a_1}(V)\to SH^{*, \le a_2}(V)$ induces a continuation map
 $\iota_{a_1,a_2;b_1,b_2}:SH^{*,\le a_1, \le b_1}(V\times W) \to SH^{*,\le a_2, \le b_2}(V\times W)$, whenever $a_1\le a_2\notin \partial(V)$ and $b_1\le b_2 \notin \cS(\partial W)$. The main theorem of \cite{oancea2006kunneth} is that 
\begin{equation}\label{eq:lim}
\varinjlim_{a,b}SH^{*,\le a, \le b}(V,W)\simeq SH^*(V\times W).
\end{equation}
Similarly, we have the product version of the pair of pants product using the splitting data. In the filtered case, the product is from $SH^{*,\le a_1,\le b_1}(V\times W) \otimes SH^{*,\le a_2,\le b_2}(V\times W)$ to $SH^{*,\le a_1+a_2,b_1+b_2}(V\times W)$ and is compatible with the product continuation maps. Using the identification in \eqref{eq:lim}, the limit of the product is the usual product structure on $SH^*(V\times W)$. Then the same argument as in \cite[Proposition 2.10]{zhou2020mathbb} yields the following, as $1+A$ is a unit in $H^*(V\times W)$.
\begin{proposition}\label{prop:kill}
	If $1+A\in H^*(V\times W)$  is mapped to zero in $\iota_{0,a,0,b}:H^*(V\times W)\to SH^{*,\le a, \le b}(V\times W)$ for $A\in \oplus_{i>0}H^{2i}(V\times W)$, then $SH^*(V\times W)=0$ and $SH^{*,\le a, \le b}_+(V\times W)\to H^{*+1}(V\times W)$ is surjective.
\end{proposition}
Finally, when defining $SH^{*,\le a, \le b}(V\times W)$, we only require the almost complex structure splits into $J_1\oplus J_2$ outside a compact set to guarantee a maximal principle. Although the cochain complex will no longer be a tensor product, but is quasi-isomorphic to the tensor product by a standard continuation map. Moreover, it makes sense to define $SH^{*,\le a,\le b}(X), SH^{*,\le a, \le b}_+(X)$ for any other (symplectically aspherical) filling $X$ of $\partial(V\times W)$ as long as we use Hamiltonians that are in the form of $H\oplus K$ outside $\partial(V\times W)$ and is $C^2$ small non-positive inside $X$ of type \eqref{II} or vanishes on $X$ (i.e.\ of type \eqref{I}, which requires another Morse function on $X$ for the construction of the cochain complex), and the analogue of Proposition \ref{prop:kill} holds for $X$.

\subsection{The standard filling $V\times \D$}
Here we consider the situation for the standard filling $V\times \D$. Let $\delta\ll 1$ be a fixed positive number, then there exists an admissible Hamiltonian $K_{1+\delta}$ on $\C=\widehat{\D}$ of type \eqref{II} with slope $1+\delta$, such that there is only one critical point $e$ at $0$ and there is only one $S^1$ family of non-constant periodic orbits $\overline{\gamma}_0$ corresponding to the shortest Reeb orbit on $\partial \D$. The symplectic action $\cA_{K_{1+\delta}}(\overline{\gamma}_0)$ is smaller than $-1$ but can be arranged to be arbitrarily close to $-1$. Then the Hamiltonian-Floer cochain complex is generated by $e,\check{\gamma}_0, \hat{\gamma}_0$ with grading $|e|=0, |\check{\gamma}_0|=-1, |\hat{\gamma}_0|=-2$. The only nontrivial differential is that $d\check{\gamma}_0=e$.

\begin{proposition}\label{prop:product}
	For any sufficiently small $\epsilon>0$, we have $SH_+^{*, \le 1+\delta, \le \epsilon}(V\times \D) \simeq H^*(V)[1]\oplus H^*(V)[2]$, and $SH_+^{*,\le 1+\delta, \le \epsilon}(V\times \D) \to H^{*+1}(V\times \D)$ is given by the projection to the first component.
\end{proposition}
\begin{proof}
	For a sufficient small $\epsilon$, $\epsilon(f-1)$ can be completed to an admissible Hamiltonian of type \eqref{II} on $V$ with slope $\epsilon$, where $f$ is the Morse function on $V$ used in \S \ref{ss:contact}. The Hamiltonian-Floer cochain complex of $\epsilon(f-1)$ is just the Morse complex of $f$. Since there is no Reeb orbit on $\partial V$ of period smaller than $\epsilon$, the Morse cohomology computes $SH^{*,\le \epsilon}(V)$. Then we can use $K_{1+\delta}\oplus \epsilon(f-1)$ to compute  $SH^{*,\le 1+\delta, \le \epsilon}(V\times \D)$. The differentials on the tensor product are given by 
	$$
	\begin{array}{rclrclrcl}
	\la d p\otimes \check{\gamma}_0,q \otimes \check{\gamma}_0 \ra &=& \la d_0 p, q\ra, &\la d p\otimes \check{\gamma}_0, q \otimes \hat{\gamma}_0 \ra &=& 0, &\la d p\otimes \check{\gamma}_0,q\otimes e \ra &=& \delta_{p,q}, \\
	\la d p\otimes\hat{\gamma}_0,q\otimes\check{\gamma}_0 \ra &=& 0, &\la d p\otimes\hat{\gamma}_0,q\otimes\hat{\gamma}_0 \ra & = &\la d_0 p, q\ra,& \la d p\otimes\hat{\gamma}_0,q\otimes e \ra & =& 0, \\
	\la d p\otimes e, q \otimes \check{\gamma}_0\ra &=&0,& \la d p \otimes e, q\otimes \hat{\gamma}_0\ra &=&0,& \la d p\otimes e, q\otimes e \ra &=&\la d_0 p, q\ra
	\end{array}
	$$
	where $d_0$ is the Morse differential of $f$ on $V$. This verifies the proposition. Moreover, the $H(V)[1]$ component is generated by check orbits and the $H(V)[2]$ component is generated by hat orbits.
\end{proof}
Moreover, the filtered positive symplectic cohomology $SH^{*,\le 1+\delta,\le \epsilon}_+(W)$ does not depend on the filling $W$ of $Y$. As the related cochain complex degenerate to two copies of Morse cochain complexes of $V$ for any filling when we push $\epsilon \to 0$ to a ``Morse-Bott" case\footnote{The Reeb dynamics on $Y_0$ are only Morse-Bott non-degenerate along $(V\backslash \partial V \times [\frac{1}{2},1])\times S^1$,
but not Morse-Bott non-degenerate along $\partial V \times \{\frac{1}{2}\} \times S^1$.} where all Reeb orbits have the same period. We will not prove this degeneration, but use a neck-stretching argument to prove this fact. 

\subsection{Neck-stretching}\label{ss:NS}
We first recall some basics of the neck-stretching procedure in \cite{bourgeois2003compactness}. We also recommend \cite[\S 2.3, 9.5]{cieliebak2018symplectic} for applications of neck-stretching in Floer theories.

We recall the setup of neck-stretching for general case following \cite[\S 3.2]{zhou2019symplectic}.  Let $(W,\lambda)$ be an exact domain and $(Y,\alpha:=\lambda|_{Y})$ be a contact type hypersurface inside $W$.\footnote{The process works for strong filling $W$ as long as $Y$ is contact hypersurface.} The hypersurface divides $W$ into a cobordism $X$ union with a domain $W'$. Then we can find a small slice $(Y\times [1-\eta,1+\eta]_r,\rd(r\alpha))$ symplectomorphic to a neighborhood of $Y$ in $W$. Assume $J|_{Y\times [1-\eta,1+\eta]_{r}}=J_0$, where $J_0$ is independent of $S^1$ and $r$ and $J_0(r\partial_{r})=R_\alpha,J_0\xi=\xi$ for $\xi:=\ker \alpha$. Then we pick a family of diffeomorphism $\phi_R:[(1-\eta)e^{1-\frac{1}{R}}, (1+\eta)e^{\frac{1}{R}-1}]\to [1-\eta,1+\eta]$ for $R\in (0,1]$ such that $\phi_1=\Id$ and $\phi_R$ near the boundary is linear with slope $1$. Then the stretched almost complex structure $NS_{R}(J)$ is defined to be $J$ outside $Y\times [1-\eta,1+\eta]$ and is $(\phi_R\times \Id)_*J_0$ on $Y_1\times [1-\eta,1+\eta]$. Then $NS_{1}(J)=J$ and $NS_{0}(J)$ gives almost complex structures on the completions $\widehat{X}$, $\widehat{W'}$ and $Y\times \R_+$, which we will refer as the fully stretched almost complex structure.

We will consider the degeneration of curves solving the Floer equation with one positive cylindrical end asymptotic to a non-constant Hamiltonian orbit of $X_H$. Here we require that $H=0$ near the contact hypersurface $Y$. Since either the orbit is simple or $J$ depends on the $S^1$ coordinate near non-simple orbits, the topmost curve in the SFT building, i.e.\ the curve in $\widehat{X}$, has the somewhere injectivity property. In particular, we can find regular $J$ on $\widehat{X}$ such that all relevant moduli spaces, i.e.\ those with point constraint from $\widehat{X}$ (used in \S \ref{s3}), or with negative cylindrical ends asymptotic to non-constant Hamiltonian orbits of $X_H$, possibly with negative punctures asymptotic to Reeb orbits of $Y$ and multiple cascades levels, are cut out transversely. We say an almost complex structure on $W$ is generic iff the fully stretched almost complex structure $NS_0(J)$ is regular on $\widehat{X}$. The set of generic almost complex structures form an open dense subset\footnote{This is because there are only finitely many moduli spaces that can have positive energy.} in the set of compatible almost complex structures  that are cylindrical convex and $S^1$, $r$ independent on $Y\times [1-\eta,1+\eta]_r$. 

For the compactification of curves in the topmost SFT level, in addition to the usual SFT building in the symplectization $Y\times \R_+$ stacked from below \cite{bourgeois2003compactness}, we also need to include Hamiltonian-Floer breakings near the cylindrical ends. In our context, since we use autonomous Hamiltonians and cascades,  we need to include curves with multiple cascades levels and their degeneration, e.g.\ $l_i=0,\infty$ in the cascades for some horizontal level $i$. A generic configuration is described in the top-right of the figure below, but we could also have more cascades levels with the connecting Morse trajectories degenerate to $0$ length or broken Morse trajectories.
\begin{figure}[H]
	\begin{center}
		\begin{tikzpicture}[scale=0.5]
		\path [fill=blue!15] (0,0) to [out=20, in=160]  (6,0) to [out=270,in=90] (6,-6) to [out=170,in=10] (0,-6) to [out=90, in=270] (0,0);
		\path [fill=red!15] (0,-6) to [out=10, in=170]  (6,-6) to [out=270,in=0] (3,-10) to [out=180,in=270] (0,-6); 
		\draw (0,0) to [out=20, in=160]  (6,0) to [out=270,in=90] (6,-6) to [out=170,in=10] (0,-6) to [out=90, in=270] (0,0);
		\draw (6,-6) to [out=270,in=0] (3,-10) to [out=180,in=270] (0,-6); 
		\draw[dashed] (0, -5.5) to [out=10, in=170] (6,-5.5);
		\draw[dashed] (0.02,-6.5) to [out=10, in=170] (5.98,-6.5);
		\draw[->] (1,-1) to (1.5,-1);
		\draw (1.5,-1) to (2,-1);
		\draw (2,-1) to [out=90, in=180] (2.5, -0.75) to [out=0, in = 90] (3,-1) to [out=270,in=0] (2.5,-1.25) to [out=180,in=270] (2,-1);
		\draw (2,-1) to [out=270,in=90] (1,-7) to [out=270,in=180] (2,-8) to [out=0,in=180](2.5,-3) to [out=0, in=90](3,-4);
		\draw (3,-1) to [out=270, in=90] (4,-4) to [out=270, in=0] (3.5,-4.25) to [out=180,in=270] (3,-4);
		\draw[dotted] (3,-4) to [out=90, in=180] (3.5,-3.75) to [out=0, in=90] (4,-4);
		\draw[->] (4,-4) to (4.25,-4);
		\draw (4.25,-4) to (4.5,-4);
		\draw (4.5,-4) to [out=90, in=180] (5,-3.75)  to [out=0, in=90] (5.5,-4) to [out=270,in=0] (5,-4.25) to [out=180,in=270] (4.5,-4);
		\draw (4.5,-4) to (4.5,-5);
		\draw (5.5,-4) to (5.5,-5);
		\draw[dotted]  (4.5,-5) to [out=90, in=180] (5,-4.75)  to [out=0, in=90] (5.5,-5);
		\draw (5.5,-5) to [out=270,in=0] (5,-5.25) to [out=180,in=270] (4.5,-5);
		\draw[->] (5.5,-5) to (5.75,-5);
		\draw (5.75,-5) to (5.8,-5);
		\node at (2.5,-2) {$u_1$};
		\node at (5,-4.6) {$u_2$};
		\end{tikzpicture}
		\hspace{1cm}
		\begin{tikzpicture}[xscale=0.5,yscale=0.7]
		\path [fill=blue!15] (0,-1) to [out=20, in=160]  (6,-1) to [out=270,in=90] (6,-6) to [out=170,in=10] (0,-6) to [out=90, in=270] (0,-1);
		\path [fill=red!15] (0,-6) to [out=10, in=170]  (6,-6) to [out=270,in=0] (3,-10) to [out=180,in=270] (0,-6); 
		\draw (0,-1) to [out=20, in=160]  (6,-1) to [out=270,in=90] (6,-6) to [out=170,in=10] (0,-6) to [out=90, in=270] (0,-1);
		\draw (6,-6) to [out=270,in=0] (3,-10) to [out=180,in=270] (0,-6);
		\draw[dashed] (0, -5.5) to [out=10, in=170] (6,-5.5);
		\draw[dashed] (0.1, -6.5) to [out=10, in=170] (5.9,-6.5);
		\draw[->] (1,-2) to (1.5,-2);
		\draw (1.5,-2) to (2,-2);
		\draw (2,-2) to [out=90, in=180] (2.5, -1.8) to [out=0, in = 90] (3,-2) to [out=270,in=0] (2.5,-2.2) to [out=180,in=270] (2,-2);
		\draw (2,-2) to [out=270,in=90] (1,-7) to [out=270,in=180] (2,-8) to [out=0,in=180](2.5,-3.5) to [out=0, in=90](3,-4);
		\draw (3,-2) to [out=270, in=90] (4,-4) to [out=270, in=0] (3.5,-4.2) to [out=180,in=270] (3,-4);
		\draw[dotted] (3,-4) to [out=90, in=180] (3.5,-3.8) to [out=0, in=90] (4,-4);
		\draw[->] (4,-4) to (4.25,-4);
		\draw (4.25,-4) to (4.5,-4);
		\draw (4.5,-4) to [out=90, in=180] (5,-3.75)  to [out=0, in=90] (5.5,-4)  to [out=270,in=0] (5,-4.25) to [out=180,in=270] (4.5,-4);
        \draw (4.5,-4) to (4.5,-5);
		\draw (5.5,-4) to (5.5,-5);
		\draw[dotted]  (4.5,-5) to [out=90, in=180] (5,-4.75)  to [out=0, in=90] (5.5,-5);
		\draw (5.5,-5) to [out=270,in=0] (5,-5.25) to [out=180,in=270] (4.5,-5);
		\draw[->] (5.5,-5) to (5.75,-5);
		\draw (5.75,-5) to (5.8,-5);
		\node at (2.5,-3) {$u_1$};
		\node at (5,-4.5) {$u_2$};
		\end{tikzpicture}
		\hspace{1cm}
		\begin{tikzpicture}[scale=0.5]
		\path [fill=blue!15] (0.5,-6) to [out=90, in=270]  (0,0) to [out=20,in=160] (6,0) to [out=270,in=90] (5.5,-6) to [out=160, in=20] (0.5,-6);
		\path [fill=purple!15] (0.5,-6.2) to [out=20,in=160] (5.5,-6.2) to [out=270,in=90] (5.5,-12) to [out=160, in=20] (0.5,-12) to [out=90, in=270] (0.5,-6.2);
		\path [fill=red!15] (0.5, -12.2) to [out=20,in=160] (5.5,-12.2) to [out=270,in=0] (3,-16) to [out=180,in=270] (0.5,-12.2);
		\draw (0.5,-6) to [out=90, in=270]  (0,0) to [out=20,in=160] (6,0) to [out=270,in=90] (5.5,-6);
		\draw [dashed] (5.5,-6) to [out=160, in=20] (0.5,-6);
		\draw[dashed] (0.5,-6.2) to [out=20,in=160] (5.5,-6.2);
		\draw (5.5,-6.2) to [out=270,in=90] (5.5,-12);
		\draw[dashed] (5.5,-12) to [out=160, in=20] (0.5,-12);
		\draw (0.5,-12) to [out=90, in=270] (0.5,-6.2);
		\draw [dashed](0.5, -12.2) to [out=20,in=160] (5.5,-12.2); 
		\draw (5.5,-12.2)to [out=270,in=0] (3,-16) to [out=180,in=270] (0.5,-12.2);
		\draw[->] (1,-1) to (1.5,-1);
		\draw (1.5,-1) to (2,-1);
		\draw (2,-1) to [out=90, in=180] (2.5, -0.75) to [out=0, in = 90] (3,-1) to [out=270,in=0] (2.5,-1.25) to [out=180,in=270] (2,-1);
		\draw (2,-1) to [out=270,in=90] (1,-4) to [out=270, in=180]  (1.5,-5.8) to [out=0,in=270](2,-5) to [out=90,in=180] (2.5,-3) to [out=0, in=90](3,-4);
		\draw (3,-1) to [out=270, in=90] (4,-4) to [out=270, in=0] (3.5,-4.25) to [out=180,in=270] (3,-4);
		\draw[dotted] (3,-4) to [out=90, in=180] (3.5,-3.75) to [out=0, in=90] (4,-4);
		\draw[->] (4,-4) to (4.25,-4);
		\draw (4.25,-4) to (4.5,-4);
		\draw (4.5,-4) to [out=90, in=180] (4.9,-3.75)  to [out=0, in=90] (5.3,-4) to [out=270,in=0] (4.9,-4.25) to [out=1800,in=270] (4.5,-4);
		\draw (4.5,-4) to (4.5,-5);
		\draw (5.3,-4) to (5.3,-5);
		\draw[dotted] (4.5,-5) to [out=90, in=180] (4.9,-4.75)  to [out=0, in=90] (5.3,-5);
		\draw (5.3,-5) to [out=270,in=0] (4.9,-5.25) to [out=1800,in=270] (4.5,-5);
        \draw[->] (5.3,-5) to (5.5,-5);
		\draw (1.5,-5.9) to [out=180,in=90] (0.8,-9) to [out=270, in=180] (1.5,-11.8) to [out=0,in=270] (2.2,-9) to [out=90, in=0] (1.5,-5.9);
		\draw (1.5,-11.9) to [out=180,in=90] (0.8,-13) to [out=270, in=180] (1.5,-14) to [out=0,in=270] (2.2,-13) to [out=90, in=0] (1.5,-11.9);

		\node at (1.5,-5.8) [circle, fill=white, draw, outer sep=0pt, inner sep=3 pt] {};

		\node at (1.5,-11.8) [circle, fill=white, draw, outer sep=0pt, inner sep=3 pt] {};
		\node at (2.5,-2) {$u^{\infty}_1$};
		\node at (5,-4.7) {$u^{\infty}_2$};
		\node at (4.5,-2) {$\widehat{X}$};
		\node at (4,-10) {$Y\times \R_+$};
		\node at (4,-14) {$\widehat{W'}$};
		\end{tikzpicture}
	\end{center}
    \caption{Neck-stretching}
    \label{fig}
\end{figure}
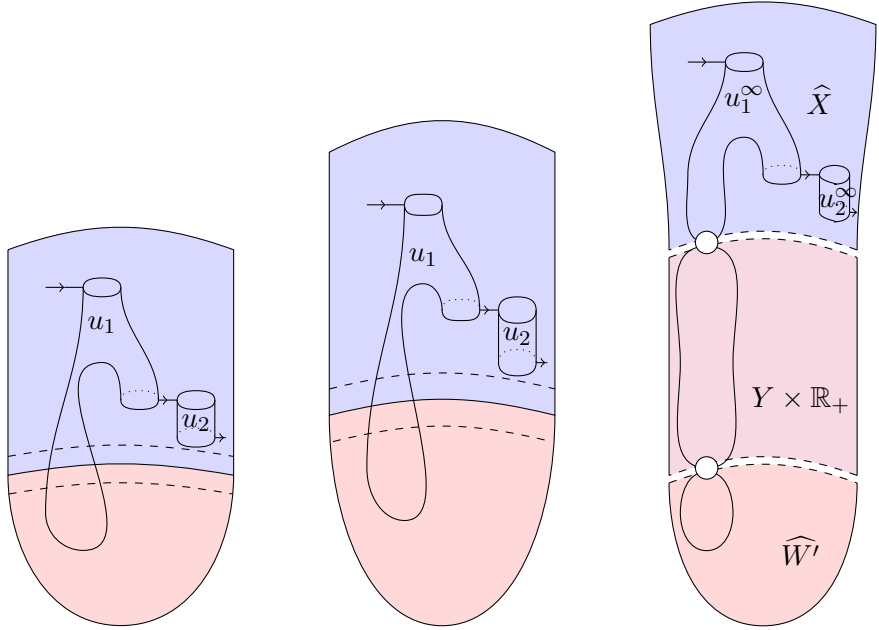
In the figure, we use $\bigcirc$ to indicate the puncture that is asymptotic to a Reeb orbit. The neck-stretching procedure allows us to understand the effect of fillings on Floer cohomology.

A useful fact from the non-negativity of energy is the following action constraint. Let $u$ be a Floer cylinder in $\widehat{X}$ with negative punctures asymptotic to a multiset $\Gamma$ of Reeb orbits (i.e.\ a set of Reeb orbits with possible duplications). Assume $\displaystyle\lim_{s\to \infty} u = x$ and $\displaystyle\lim_{s\to -\infty} u = y$,  then we have
\begin{equation}\label{eqn:positive}
\cA_H(y)-\cA_H(x)-\sum_{\gamma\in \Gamma}\int \gamma^*\lambda\ge 0
\end{equation}
In \S \ref{s3}, we need to consider Floer cylinders in $\widehat{X}$ such that  $\displaystyle\lim_{s\to -\infty} u$ converges to a point (where $H=0$), then we have 
\begin{equation}\label{eqn:positive2}
-\cA_H(x)-\sum_{\gamma\in \Gamma}\int \gamma^*\lambda\ge 0
\end{equation}

\begin{proposition}\label{prop:positive}
	Let $W$ be a symplectically aspherical filling of $\partial (V \times \D)$, then  $SH_+^{*,\le 1+\delta, \le \epsilon}(W) \simeq H^*(V)[1]\oplus H^*(V)[2]$ for any sufficiently small $\epsilon$. 
\end{proposition}
\begin{proof}
The action difference between any two generators in $SH_+^{\le 1+\delta, \le \epsilon}(W)$ is very small for sufficient small $\epsilon$ when we use Hamiltonian $K_{1+\delta}\oplus \epsilon(f-1)$ outside $\partial{W}=\partial (V\times \D)$. W.O.L.G., i.e.\ up to rescaling, we can assume $Y_0$ is contained in the exact neighborhood (i.e.\ where the Liouville vector field exists) of $\partial W$. We can apply neck-stretching along $Y_0$. Since all Reeb orbits have period at least $1$, there is no breaking for a fully stretched almost complex structure by  the action constraint \eqref{eqn:positive}. Therefore all relevant moduli spaces are contained outside $Y_0$ for a  sufficiently stretched almost complex structure, i.e.\  $SH_+^{*,\le 1+\delta, \le \epsilon}(W)$ is independent of the filling.
\end{proof}

In Proposition \ref{prop:product}, the splitting $SH^{*,\le 1+\delta,\le \epsilon}_+(V\times \D)=H^*(V)[1]\oplus H^*(V)[2]$ is given by check and hat orbits. Next, we explain that we have the same splitting for any filling. Since in our situation, only simple Reeb orbits are considered. Therefore we can require our almost complex structure to be time-independent and still have all the transversality properties \cite[Proposition 3.5]{bourgeois2009symplectic}. Then the moduli spaces of Floer cylinders considered in the positive cochain complex will have a free $S^1$ action. Therefore, there is no rigid cascade from a hat orbit to a check orbit because of the free $S^1$ action on Floer cylinders, as rotating any Floer cylinder a bit is still a cascade from the hat orbit to the check orbit. Let $\check{C}_+$ and $\hat{C}_+$ denote the complexes generated by check orbits and hat orbits respectively, then there is a short exact sequence of complexes $0\to \hat{C}_+\to C_+\to \check{C}_+\to 0$. The $S^1$ equivariant transversality argument holds for continuation maps when using an $S^1$-independent almost complex structure. Hence the continuation map induces an morphism between the short exact sequences. We define $\check{SH}^{*,\le 1+\delta,\le \epsilon} := H^*(\check{C}_+)$ and  $\hat{SH}^{*,\le 1+\delta,\le \epsilon} := H^*(\hat{C}_+)$. Therefore we have the following.
\begin{proposition}\label{prop:checkhat}
		Let $W$ be a symplectically aspherical filling of $\partial (V \times \D)$. For any sufficiently small $\epsilon$, we have a short exact sequence $0 \to \hat{SH}_+^{*,\le 1+\delta, \le \epsilon}(W) \to SH_+^{*,\le 1+\delta, \le \epsilon}(W) \to \check{SH}_+^{*,\le 1+\delta, \le \epsilon}(W) \to 0$, which is isomorphic to $0\to H^*(V)[2]\to H^*(V)[1]\oplus H^*(V)[2] \to H^*(V)[1]\to 0$. Moreover, the connecting map $ SH_+^{*,\le 1+\delta, \le \epsilon}(W)\to H^{*+1}(W)$ factors through $ SH_+^{*,\le 1+\delta, \le \epsilon}(W) \to \check{SH}_+^{*,\le 1+\delta, \le \epsilon}(W)$.
\end{proposition}
\begin{proof}
	For a sufficiently stretched $S^1$-independent almost complex structure, the short exact sequence $0\to \hat{C}_+\to C_+\to \check{C}_+\to 0$ is the same for $V\times \D$ and $W$ by the action argument in Proposition \ref{prop:positive}.  For $V\times \D$, it is important to note that a stretched almost complex structure does not split. However, the continuation map from a splitting almost complex structure to a sufficiently  stretched almost complex structure induces a morphism between the short exact sequences using $S^1$-independent almost complex structures. It is clear by action reasons, the induced continuation maps are isomorphisms (upper triangular w.r.t.\ to the filtration from the values of $f$, i.e.\ the filtration from the symplectic action/contact action, and are identity on the diagonal) on $\check{C}_+$ and $\hat{C}_+$.  As a consequence, the induced long exact sequence is isomorphic to the one from the splitted $J$ in Proposition \ref{prop:product}, whose long exact sequence splits, i.e.\ gives rise to the short exact sequence in the claim. The last claim follows from the  $S^1$-equivariant transversality, as there is no differential (no rigid cascades) from hat orbits to constant orbits by the free $S^1$ action.
\end{proof}

\subsection{A continuation map}
In \S \ref{s3}, we need to stretch along the contact hypersurface $Y_{\epsilon}$ to prove certain independence of fillings. Since we do not have $c_1(Y)=0$, the Fredholm index of a curve also depends on the relative homology class. We need to show that the relative homology class is always trivial and for this we will use the contact energy (see Proposition \ref{prop:homology}), hence we had better use the symplectic cohomology with admissible Hamiltonians in \S \ref{ss:sympcoh}. Therefore we need a continuation map relating $SH_+^{*,\le a, \le b}(V\times W)$ and $SH_+^{*\le c}(V\times W)$. This was constructed in \cite{oancea2006kunneth} for the proof of the K\"unneth formula, we recall an adapted version for the case in this paper. Let $H_{1+\epsilon+2\delta}$ be an admissible Hamiltonian on $\widehat{V}\times \C$ of type \eqref{I} of the contact hypersurface $Y_{\epsilon}$ with slope $1+\epsilon+2\delta$.  
\begin{proposition}\label{prop:nice}
	For any sufficiently small $\epsilon$, we can arrange that $K_{1+\delta}\oplus \epsilon(f-1)$ is pointwise no greater than $H_{1+\epsilon+2\delta}$ on $\widehat{V}\times \C$. Moreover, when $\epsilon$ is sufficiently small, for any critical points $p,q$ of $f$, such that $f(p)>f(q)$, we have 
	$$\cA_{H_{1+\epsilon+2\delta}}(\overline{\gamma}_p)> \cA_{K_{1+\delta}\oplus \epsilon(f-1)}(p\otimes\overline{\gamma}_0) > \cA_{H_{1+\epsilon+2\delta}}(\overline{\gamma}_q)>
	\cA_{K_{1+\delta}\oplus \epsilon(f-1)}(q\otimes\overline{\gamma}_0)$$
\end{proposition}
\begin{proof}
	We first prove the claim for the extreme case, then we will argue that we can perturb the extreme case to get admissible choices of $K_{1+\delta}$ and $H_{1+\epsilon+2\delta}$. The extreme case is that $K_{1+\delta}=0$ on $\D$ and is linear of slope $1$ w.r.t.\ $\rho^2$ outside $\D$, then picks up the slope $1+\delta$ outside a very large compact set. The Hamiltonian orbit is considered as placed at $\partial \D$. $H_{1+\epsilon+2\delta}$ is $0$ inside $Y_{\epsilon}$ and has slope $1+\epsilon+2\delta$ outside $Y_{\epsilon}$, the Hamiltonian orbit $\overline{\gamma}_p$ is considered as on $Y_{\epsilon}$. Then we have
	$$\cA_{H_{1+\epsilon+2\delta}}(\overline{\gamma}_p)=-\frac{1+\epsilon}{1+\epsilon f(p)}, \quad \cA_{K_{1+\delta}\oplus \epsilon(f-1)}(p\otimes \overline{\gamma}_0)=-1+\epsilon(f(p)-1).$$
	Note that 
	$$-\frac{1+\epsilon}{1+\epsilon f(p)}-(-1+\epsilon(f(p)-1))=\epsilon(1-f(p))(1-\frac{1}{1+\epsilon f(p)})$$
	which is non-negative as $0\le f(p)<1$ and is zero if and only if when $p$ is the minimum of $f$. Hence we have $$\cA_{H_{1+\epsilon+2\delta}}(\overline{\gamma}_p) \ge \cA_{K_{1+\delta}\oplus \epsilon(f-1)}(p\otimes\overline{\gamma}_0)$$
	for all critical points $p$ and the only case when the equality holds is 
	when $p$ is the minimum point. Moreover, recall from \S \ref{ss:contact}, if $\ind(q)=2n-k>0$, then $f(q)=\frac{1}{k+1}$. Now let $p$ be another critical point of $f$ with $f(p)=\frac{1}{k}$, i.e.\ $\ind(p)=2n+1-k$. Then we have
	\begin{eqnarray*}
	\cA_{K_{1+\delta}\oplus \epsilon(f-1)}(p\otimes \overline{\gamma}_0)-\cA_{H_{1+\epsilon+2\delta}}(\overline{\gamma}_q) & = & -1+\epsilon(\frac{1}{k}-1)+\frac{1+\epsilon}{1+\epsilon\frac{1}{k+1}} \\
	& = & \epsilon \frac{1-\frac{1}{k+1}}{1+\epsilon \frac{1}{k+1}}+\epsilon (\frac{1}{k}-1)\\
	& = & \frac{\epsilon}{1+\epsilon \frac{1}{k+1}} \left(1-\frac{1}{k+1}+(\frac{1}{k}-1)(1+\epsilon\frac{1}{k+1}) \right)\\
	& = & \frac{\epsilon}{1+\epsilon \frac{1}{k+1}} \frac{1-(k-1)\epsilon}{k(k+1)}
	\end{eqnarray*}
	Therefore when $\epsilon<\frac{1}{k-1}$, we have $$\cA_{H_{1+\epsilon+2\delta}}(\overline{\gamma}_q)<\cA_{K_{1+\delta}\oplus \epsilon(f-1)}(p\otimes \overline{\gamma}_0).$$
	When $q$ is the minimum point, hence $\cA_{H_{1+\epsilon+2\delta}}(\overline{\gamma}_q)=-1-\epsilon=\cA_{K_{1+\delta}\oplus \epsilon(f-1)}(q\otimes \overline{\gamma}_0)$, which is smaller than $\cA_{K_{1+\delta}\oplus \epsilon(f-1)}(p\otimes \overline{\gamma}_0)$ for any $p$ that is not the minimum. Therefore we have proven for the extreme case and $\epsilon$ sufficiently small ($<\frac{1}{2n}$) that 
		\begin{equation}\label{eqn:order}
		\cA_{H_{1+\epsilon+2\delta}}(\overline{\gamma}_p) > \cA_{K_{1+\delta}\oplus \epsilon(f-1)}(p\otimes\overline{\gamma}_0) > \cA_{H_{1+\epsilon+2\delta}}(\overline{\gamma}_q) \ge \cA_{K_{1+\delta}\oplus \epsilon(f-1)}(q\otimes\overline{\gamma}_0)
		\end{equation}
	for any critical points $p,q$ with $f(p)>f(q)$, with equality holds only for $q$ is the minimum point.
	
	We claim if $\epsilon$ is small enough, then $K_{1+\delta}\oplus \epsilon(f-1)\le H_{1+\epsilon+2\delta}$ pointwise. 
	
	We first claim that inside $Y_{\epsilon}$, we have $K_{1+\delta}\oplus \epsilon(f-1)\le H_{1+\epsilon+2\delta}$. Since on the sub-domain of $V\times \D$ that is bounded by $Y_{\epsilon}$, we have that $H_{1+\epsilon+2\delta}=0$, $K_{1+\delta}=0$ and $ \epsilon(f-1)\le 0$, hence the claim holds on that sub-domain. Then on the domain outside $V\times \D$ and inside $Y_{\epsilon}$, we have that $H_{1+\epsilon+2\delta}=0$, $K_{1+\delta}\le \frac{1+\epsilon}{1+\epsilon f}-1=\frac{\epsilon (1-f)}{1+\epsilon f}$. Since $\frac{\epsilon(1-f)}{1+\epsilon f}+\epsilon(f-1) = -\epsilon^2f\frac{1-f}{1+\epsilon f}\le 0$, the claim holds on that sub-domian. 
	
	Then we will show that for any point on $Y_{\epsilon}$, the inequality holds along the (positive) flow of the Liouville vector field $X_{\lambda}+\frac{1}{2}\rho \partial_{\rho}$. Since the angular coordinate on $\C$ does not matter, we choose $(x,\rho^2)$ for $x\in V,\rho \in \R^+$ to represent the point. Then after time $t$ flow of $X_{\lambda}+\frac{1}{2}\rho \partial_{\rho}$, the point is $(\phi_t(x), \rho^2 e^t)$, where $\phi_t$ is the flow of $X_{\lambda}$. We separate $Y_{\epsilon}$ into the graph of $\rho^2=\frac{1+\epsilon}{1+\epsilon f}$ and the graph of $r=g_{\epsilon}(\rho^2)$ as in \S \ref{ss:contact}.
	
	On the first graph, we have $\rho^2=\frac{1+\epsilon}{1+\epsilon f(x)},$
	then we have 
	$$H_{1+\epsilon+2\delta}(\phi_{t}(x),\frac{1+\epsilon}{1+\epsilon f(x)}e^t)=(1+\epsilon+2\delta)(e^t-1),$$ and 
	$$(K_{1+\delta}\oplus \epsilon(f-1))(\phi_t(x),\frac{1+\epsilon}{1+\epsilon f(x)}e^t)=\frac{1+\epsilon}{1+\epsilon f(x)}e^t-1+\epsilon(f\circ \phi_t(x)-1),$$ 
	while the point is still on the domain where the slope of $K_{1+\delta}$ is $1$. Therefore we compute
	\begin{equation}\label{eqn:diff1}
	\frac{\rd}{\rd t}\left(H_{1+\epsilon+2\delta}(\phi_{t}(x),\frac{1+\epsilon}{1+\epsilon f(x)}e^t)-\left(K_{1+\delta}\oplus \epsilon(f-1)\right)(\phi_t(x),\frac{1+\epsilon}{1+\epsilon f(x)}e^t)\right)\ge 2\delta e^t-\epsilon X_{\lambda}(f\circ \phi_t(x))
	\end{equation}
	If $\phi_t(x)\notin V$, assume $\phi_{t_0}(x)\in \partial V$. Since $(x,\rho^2)$, by assumption, is on the graph of  $\rho^2=\frac{1+\epsilon}{1+\epsilon f}$, we have $x$ is contained in $V$. As a consequence, we have $t_0>0$. Since $\partial_r f = 1$ outside $V$, where $r=e^{t-t_0}$, then $X_{\lambda}(f\circ \phi_t(x))=e^{t-t_0}\le e^t$. Then for $\epsilon$ small enough ($<2\delta$), we have \eqref{eqn:diff1} is positive. When $K_{1+\delta}$ starts to pick up the slope of $1+\delta$ for $t$ very big. The \eqref{eqn:diff1} decrease at most $\frac{\delta(1+\epsilon)}{1+\epsilon f(x)}e^t$, which will not change the sign for $\epsilon \ll  1$. 
	
	On the graph of $r=g_{\epsilon}(\rho^2)$, we use the $(r,\rho^2)=(g_{\epsilon}(\rho^2),\rho^2)$ coordinate. After time $t$, the new coordinate is $(g_{\epsilon}(\rho^2)e^t,\rho^2e^t)$. Then we can compute 
	\begin{equation}\label{eqn:diff2}
	\frac{\rd}{\rd t}\left(H_{1+\epsilon+2\delta}(g_{\epsilon}(\rho^2)e^t,\rho^2e^t)-\left(K_{1+\delta}\oplus \epsilon(f-1)\right)(g_{\epsilon}(\rho^2)e^t,\rho^2e^t)\right)\ge (1+\epsilon+2\delta-\rho^2) e^t-\epsilon\frac{\rd}{\rd t} \left(f(g_{\epsilon}(\rho^2)e^t)\right),
	\end{equation}
	while the point is on the domain where the slope of $K_{1+\delta}$ is still $1$. Since $\rho^2\le \frac{1+\epsilon}{1+\epsilon f(\frac{1}{2})}<1+\epsilon$, $\frac{1}{2}\le g_{\epsilon}(\rho^2) \le \frac{3}{4}$ and $\partial_r f\le 1$, we have \eqref{eqn:diff2} $\ge 2\delta e^t-\epsilon e^t$. Therefore for $\epsilon$ small enough, \eqref{eqn:diff2} is positive. When $K_{1+\delta}$ starts to pick up the slope $1+\delta$, \eqref{eqn:diff2} decreases at most $\rho^2\delta e^t \le (1+\epsilon)\delta e^t$, which will not change the sign.
	
	To sum up, in the extreme case, we have that $K_{1+\delta}\oplus \epsilon(f-1)$ is not greater than $H_{1+\epsilon+2\delta}$, with the equality holds only on the sub-domain  of $V\times \D $ bounded by $Y_{\epsilon}$. Then we modify $K_{1+\delta}$ to a smooth function, such that it picks up the first Reeb orbit shortly after $\rho^2=1$, then maintains a slope slightly bigger than $1$ for a very long time, then gradually picks up the slope till it is $1+\delta$. Then the modified $K_{1+\delta}$ is strictly smaller than the extreme $K_{1+\delta}$ outside $\D$. Such modification will decrease the symplectic action by an arbitrarily small amount, then \eqref{eqn:order} becomes strict. Then we can perturb $H_{1+\epsilon+2\delta}$ to a smooth one, which is pointwise no less than $K_{1+\delta}\oplus \epsilon(f-1)$. The strict order in \eqref{eqn:order} can be preserved under such a small change.
\end{proof}

Following \cite{oancea2006kunneth}, we can build a continuation map from $C^*(K_{1+\delta}\oplus \epsilon(f-1))\to C^*(H_{1+\epsilon+2\delta})$ using a decreasing homotopy of Hamiltonians, which also induces a continuation map for the positive cochain complexes. 
\begin{remark}
	Strictly speaking, one needs to modify $K_{1+\delta}\oplus \epsilon(f-1)$ outside a large compact set before interpolating the geometric data to guarantee that a maximal principle holds. This procedure will create many periodic orbits with arbitrarily large positive symplectic action, hence $C^*(K_{1+\delta}\oplus \epsilon(f-1))$ is a quotient complex and the continuation map does not see those extra generators, since the continuation map increases symplectic actions and the symplectic action of orbits of $H_{1+\epsilon+2\delta}$ are bounded above. See \cite{oancea2006kunneth} for details of the construction of this continuation map, but note that our convention of symplectic action is different from \cite{oancea2006kunneth} by a sign.
\end{remark}
For $C_+(K_{1+\delta}\oplus\epsilon(f-1))$ and $C_+(H_{1+\epsilon+2\delta})$, we have a filtration induced by the symplectic action. Since our choice of $f$ is self-indexing, the filtration $F_k\supset F_{k+1}$ is the following,
$$F_k:=\la p\otimes\check{\gamma}_0,p\otimes \hat{\gamma}_0|\ind(p)\ge k\ra, \text{ or }\la \check{\gamma}_p,\hat{\gamma}_p|\ind(p)\ge k\ra$$
for $C_+(K_{1+\delta}\oplus \epsilon(f-1))$ and $C_+(H_{1+\epsilon+2\delta})$ respectively. We also have filtrations $\check{F}_k,\hat{F}_k$ on $\check{C}_+$ and $\hat{C}_+$ compatible with the short exact sequence. The significance of Proposition \ref{prop:nice} is that the continuation map will preserve the filtration. The purpose of such filtration is to substitute the $\Z$ grading on symplectic cohomology, which may not exist if the first Chern class of the filling does not vanish. The following result follows from Proposition \ref{prop:nice} and neck-stretching as in Proposition \ref{prop:positive}.

\begin{proposition}\label{prop:filtration}
For any sufficient small $\epsilon$, the continuation map $H^{*}(C_+(K_{1+\delta}\oplus \epsilon(f-1))) \to H^{*}(C_+(H_{1+\epsilon+2\delta}))$ is independent of the filling $W$ of $Y$. Moreover, the continuation map preserves the filtration and the short exact sequences of check and hat orbits. 
\end{proposition}

However, there is an unsatisfying fact about Proposition \ref{prop:nice}, i.e.\ when $\epsilon \to 0$, $H_{1+\epsilon+2\delta}$ is forced to be only $C^0$ convergent to the ``ideal" Hamiltonian, which is zero on $Y_{0}$ and is linear with slope $1+\epsilon+2\delta$ outside $Y_0$. This poses analytical problems later (\S \ref{ss31}) in the compactness argument for $\epsilon \to 0$. The following proposition remedies the issue. Recall that in SFT, we have the notion of contact energy $\int u^*\alpha$ for curves in the symplectization $(Y\times \R_+, \rd(r\alpha))$ and the energy is non-negative and it is zero if and only if $u$ a trivial solution over some Reeb trajectory \cite{bourgeois2003compactness}. In the context of Hamiltonian-Floer theory, if we use Hamiltonians of type \eqref{I} or \eqref{II}, then $X_H$ is parallel to the Reeb vector field outside $Y$. Assume we pick the almost complex structure to be cylindrical convex outside $Y$, i.e.\ $\widehat{\lambda}\circ J = \rd r$ and compatible with $\rd\widehat{\lambda}$. In this case, we can still control the contact energy for the portion outside $Y$, $\int_{u^{-1}(\widehat{W}\backslash W)} u^*(\lambda|_Y)$, which is again non-negative. When the contact energy is zero, $u|_{u^{-1}(\widehat{W}\backslash W)}$ is contained in $\gamma \times [1,\infty)$, where $\gamma$ is a Reeb trajectory on $Y$.

\begin{proposition}\label{prop:better}
	Let $H^1_{1+\epsilon+2\delta}, H^2_{1+\epsilon+2\delta}$ be two admissible Hamiltonian of type \eqref{I} with slope $1+\epsilon+2\delta$ and $H^1_{1+\epsilon+2\delta}\le H^2_{1+\epsilon+2\delta}$, then there is a continuation map from $C_+^*(H^1_{1+\epsilon+2\delta})$ to $C_+^*(H^2_{1+\epsilon+2\delta})$, which is an isomorphism and preserves the filtration and does not depend on the filling.
\end{proposition}
\begin{proof}
	Let $H_s$ be the obvious decreasing homotopy from $H^2_{1+\epsilon+2\delta}$ to $H^1_{1+\epsilon+2\delta}$ such that for each $s$, $H_s$ only depends on $r$ and $X_{H_s}$ is parallel to the Reeb vector field. Then we can pick a regular almost complex structure such that it is cylindrical outside $Y_{\epsilon}$, such that all relevant moduli spaces stay outside $Y_{\epsilon}$. This is again by neck-stretching as in Proposition \ref{prop:positive} and the regularity is possible since we only consider simple orbits. Then in such a special case, $X_{H_s}$ is parallel to the Reeb vector field everywhere, the contact energy $\int_{\R\times S^1} u^*(\pi^*(\lambda|_{Y_{\epsilon}}))$ is non-negative for any Floer solution $u$, where $\pi$ is the projection from the positive symplectization $Y_{\epsilon}\times [1,\infty)$ (this is exactly symplectomorphic to the sub-domain of $\widehat{W}$ outside $Y_\epsilon$) to $Y_{\epsilon}$. As a consequence, the continuation preserves the contact action filtration, which is just the filtration from the period of Reeb orbits, or equivalently the Morse index filtration $F_k$. The contact energy is zero iff it is a reparameterization of a trivial cylinder, which is transverse. Therefore the continuation map is identity on the diagonal. This finishes proof.
\end{proof}
In other words, although $H_{1+\epsilon+2\delta}$ from Proposition \ref{prop:nice} does not converge as smooth functions for $\epsilon\to 0$, the associated cochain complexes do ``converge". More precisely, as a consequence of Proposition \ref{prop:better}, we can find a smooth family of functions $\widetilde{H}_{1+\epsilon+2\delta}$ for $\epsilon \ge 0$, such that each $\widetilde{H}_{1+\epsilon+2\delta}$ is admissible of type \eqref{I} of slope $1+\epsilon+2\delta$ and is pointwise no larger than the $H_{1+\epsilon+2\delta}$ constructed in Proposition \ref{prop:nice}.\footnote{As the ``ideal" limit $\displaystyle \lim_{\epsilon\to 0} H_{1+\epsilon+2\delta}$ is greater than or equal to any admissible Hamiltonian of type \eqref{I} with slope $1+\epsilon+2\delta$.} Then for $\epsilon>0$ small, the following map preserves the filtration, is compatible with the short exact sequence and is independent of the filling,
\begin{equation}\label{eqn:continuation}
    C^*_+(K_{1+\delta}\oplus \epsilon(f-1))\to C^*_+(H_{1+\epsilon+2\delta})\to C^*_+(\widetilde{H}_{1+\epsilon+2\delta}),
\end{equation}
where the second map is the inverse of the continuation map in Proposition \ref{prop:better}.

\section{Homology cobordism}\label{s3}
In this section we prove that the composition $SH_+^{*,\le 1+\delta, \le \epsilon}(W) \to SH_+^{*,\le 1+\epsilon+2\delta}(W)\to H^{*+1}(W)\mapsto H^{*+1}(Y)$ is independent of symplectically aspherical/Calabi-Yau fillings, which,  combined with the case for the standard filling $V\times \D$, will yield the proof of Theorem \ref{thm:main}. Here the first is map is the continuation map \eqref{eqn:continuation}. We separate the proof into the symplectically aspherical case and the Calabi-Yau case. The symplectically aspherical case is more involved due to the missing of a global $\Z$ grading. But the action filtration $F_k$ will serve as a substitute of the grading.
\subsection{ The symplectically aspherical case}\label{ss31}
Let $\widehat{W\backslash Y_{\epsilon}}$ denote the completion in the negative direction of the domain in $\widehat{W}$ outside $Y_{\epsilon}$, i.e.\ $\widehat{X}$ in \S \ref{ss:NS}. Then $\widetilde{H}_{1+\epsilon+2\delta}$ is well-defined on  $\widehat{W\backslash Y_{\epsilon}}$. We consider the moduli space $\cM_{\overline{\gamma}_p,\gamma_q}(\widetilde{H}_{1+\epsilon+2\delta})$, which is the compactification of the following
$$\left\{u:\C\backslash \{*\}\to \widehat{W\backslash Y_{\epsilon}}|\partial_s u+J(\partial_tu-X_{\widetilde{H}_{1+\epsilon+2\delta}})=0, \lim_{s\to \infty} u(t) = \overline{\gamma}_p(t+\theta), \lim_{\to *} u = (\gamma_q,-\infty) \right\}/\R$$
where $\gamma_q$ is a Reeb orbit on $Y_{\epsilon}$ which is the asymptotic of the {\em free} negative puncture $*$ and $\R$ is the translation (which moves the puncture) action on $\C$. Note that $\overline{\gamma}_p,\gamma_q$ are both contractible in $Y_{\epsilon}$ with a standard bounding disk in the standard filling $V\times \D$, which can be pushed into the boundary. For $u\in \cM_{\overline{\gamma}_p,\gamma_q}(\widetilde{H}_{1+\epsilon+2\delta})$, we use $[u]$ to denote the class in $H_2(Y)$ given by capping off $u$ with the two standard disks. $u$ is called homologically trivial iff $[u]=0$. The following is based on the compactness results in \cite{bourgeois2002morse,bourgeois2009symplectic}.

\begin{proposition}\label{prop:homology}
	For $\epsilon$ sufficiently small,  all curves in $\cM_{\overline{\gamma}_p,\gamma_q}(\widetilde{H}_{1+\epsilon+2\delta})$ for any $p,q$ must be homologically trivial.
\end{proposition}
\begin{proof}
	Assume otherwise, we have $u_{\epsilon} \in \cM_{\overline{\gamma}_p,\gamma_q}(\widetilde{H}_{1+\epsilon+2\delta})$ which are not homologically trivial for $\epsilon \to 0$. Then $u_{\epsilon}$ converges to a cascade as a hybrid of \cite{bourgeois2002morse} (for the symplectization end) and \cite{bourgeois2009symplectic} (for the Hamiltonian end). The only place which can contribute nontrivial homology is the middle Floer cylinder. But in the case when $\epsilon=0$, the contact energy $\int u^*\circ \pi^*(\lambda|_{Y_0})$ must be zero, where $\pi$ is the projection from the positive symplectization $\widehat{W\backslash Y_0}Y_{0}\times [1,\infty)$ to $Y_{0}$. Hence the middle Floer cylinder is a reparametrization of a trivial cylinder, which is homologically trivial, contradiction.
\end{proof}

Although $c_1(Y)$ is not zero as long as $c_1(V)\ne 0$, as we will see below, Proposition \ref{prop:homology} implies that the relevant moduli spaces of holomorphic  curves do not pick up nontrivial first Chern classes from $V$, which allows us to compute the dimension after neck-stretching. We are interested in the cochain map $\delta_{\partial}:C_+^*(\widetilde{H}_{1+\epsilon+2\delta})\to C_0(Y)$, which computes the map $SH^{*,\le 1+\epsilon+2\delta}_+(W)\to H^{*+1}(W)\to H^{*+1}(Y)$. For this, we pick a Morse function $h$ on $Y_{\epsilon}$ with a generic metric. Then following \cite[\S 3]{zhou2019symplectic}, we know that cochain map is define by counting the following configuration.

\begin{figure}[H]\label{fig:2}
	\begin{tikzpicture}
	\draw (0,0) to [out=90, in = 180] (0.5, 0.25) to [out=0, in=90] (1,0) to [out=270, in=0] (0.5,-0.25)
	to [out = 180, in=270] (0,0) to (0,-1);
	\draw[dotted] (0,-1) to  [out=90, in = 180] (0.5, -0.75) to [out=0, in=90] (1,-1);
	\draw (1,-1) to [out=270, in=0] (0.5,-1.25) to [out = 180, in=270] (0,-1);
	\draw (1,0) to (1,-1);
	\draw[->] (1,-1) to (1.25,-1);
	\draw (1.25,-1) to (1.5,-1);
	\draw (1.5,-1) to [out=90, in = 180] (2, -0.75) to [out=0, in=90] (2.5,-1) to [out=270, in=0] (2,-1.25)
	to [out = 180, in=270] (1.5,-1) to (1.5,-2);
	\draw (2.5,-2) to [out=270, in=0] (2,-2.5) to [out = 180, in=270] (1.5,-2);
	\draw (2.5,-1) to (2.5,-2);
	\draw[->] (-0.5,0) to (-0.25,0);
	\draw (-0.25,0) to (0,0);
	\draw[->] (2,-2.5) to (2.5,-2.5);
	\draw (2.5,-2.5) to (3,-2.5);
	\node at (2.5, -2.8) {$\nabla h$};  
	\node at (0.5,-0.5) {$u_1$};
	\node at (2, -1.5) {$u_2$};
	\end{tikzpicture}
	\caption{$\delta_{\partial}$ from $2$ level cascades}
\end{figure}
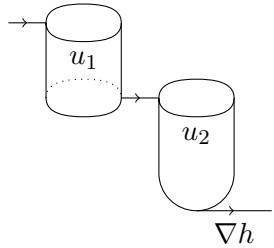
By the same $S^1$ equivariant transversality argument as before, any solution from a hat orbit is never rigid.  For $p\in \Crit(f),q\in \Crit(h)$, we use $\cM^{\partial}_{\check{\gamma}_p,q}$ to denote the compactified moduli space.
\begin{proposition}\label{prop:stretched}
Assume $\epsilon$ is sufficiently small. Let $W$ be a symplectic aspherical filling of $Y_\epsilon$, then the cochain morphism $\delta_{\partial}:C_+^*(\widetilde{H}_{1+\epsilon+2\delta})\to C_0(h)$ has the following property for some choice of $J$.
\begin{enumerate}
	\item $\delta_{\partial}(\hat{\gamma}_p)=0$.
	\item $\delta_{\partial}(\check{\gamma}_p) = a+b$ with $\ind(a)=\ind(p)$ and $\ind(b)>\ind(a)=\ind(p)$\footnote{Here $b$ may not have pure degree, then that $\ind(b)>\ind(a)$ means that any component of $b$ has larger Morse index compared to $a$.}, moreover $a$ does not depend on the filling,
\end{enumerate}
\end{proposition}
\begin{proof}
The first property follows from $S^1$-equivariant transversality. In order to prove the second claim, we need to prove the following two properties.
\begin{enumerate}
	\item $\la \delta_{\partial} \check{\gamma}_p, q \ra = 0$ if $\ind(p) > \ind(q)$.
	\item $\la \delta_{\partial} \check {\gamma}_p,q\ra$ is independent of the filling if $\ind(p)=\ind(q)$.
\end{enumerate} 
Note that the first property holds for the standard filling $V\times \D$. For both claims, it is equivalent to prove that
$\la \delta_{\partial} \check {\gamma}_p,q\ra$ is independent of the filling if $\ind(p)\ge \ind(q)$.  Note that $\la \delta_{\partial} \check{\gamma}_p, q \ra = \# \cM^{\partial}_{\check{\gamma}_p,q}$, we claim $\cM^{\partial}_{\check{\gamma}_p,q}$ is contained outside $Y_{\epsilon}$ for sufficiently stretched almost complex structure as long as $\ind(p)\ge \ind(q)$. Assume otherwise, in the fully stretched situation, the top curve will have multiple negative punctures asymptotic to Reeb orbits on $Y_{\epsilon}$. Then by the action constraint \eqref{eqn:positive2}, there is exactly one negative puncture with the asymptotic Reeb orbit $\gamma_w$ for a critical point $w$ of $f$. The Conley-Zehnder index of $\gamma_{w}$ using the obvious disk is $n-\ind(w)+2$ following \cite[Theorem 6.3]{zhou2019symplectic}. Therefore the Floer part is a curve in $\cM_{\overline{\gamma}_p,\gamma_w}(\widetilde{H}_{1+\epsilon+2\delta})$ in Proposition \ref{prop:homology}, which has trivial homology class. Then by Proposition \ref{prop:homology}, the virtual dimension of such configuration is 
$$\ind(q)-\ind(p) - (2n-\ind (w))<0, \text{ when } \ind(p)\ge \ind (q).$$
As a consequence, there is no such curve. We reach at a contradiction.
\end{proof}

\begin{remark}
In the case when $V$ is Weinstein and $c_1(V)=0$, the SFT degree ($\mu_{CZ}+n-3$) of $\gamma_p$ is bounded below by $n$. However, for general $V$ with $H^{2n-1}(V) \ne 0$ and $c_1(V)=0$, the SFT degree of $\gamma_p$ is bounded below by $1$. From the proof of Proposition \ref{prop:stretched}, we see that $H^{2n}(V)\ne 0$ is exactly the borderline case for the argument fails. The proof of Proposition \ref{prop:stretched} shows that even though there might be interesting augmentations, the augmentation does not affect the part we are interested in. The situation changes dramatically when $V$ becomes closed, i.e.\ if we consider negative line bundles over a symplectically aspherical manifold $V$. Then by \cite{oancea2008fibered}, the symplectic cohomology is zero. But now the augmentation to the Reeb orbit corresponding to $H^{2n}(V)$ plays an essential role. And the elimination pattern is completely different, in particular, $1$ is only killed after we include the $n$th-multiple covers of the simple Reeb orbits, see \cite{ritter2014floer}.	
\end{remark}

\begin{proof}[Proof of Theorem \ref{thm:main} for the symplectically apherical case] We first assume $W$ is exact for simplicity. Combining Proposition \ref{prop:filtration}, \ref{prop:better} and \ref{prop:stretched} together, we know that $\Phi:C_+^*(K_{1+\delta}\oplus\epsilon(f-1))\to C_+^*(H_{1+\epsilon+2\delta})\to C_+^*(\widetilde{H}_{1+\epsilon+2\delta})\to C^{*+1}(Y)$ preserves the index filtration, and the map on the  associated graded group is independent of fillings. Then on the associated graded group of cohomology, the induced map $\oplus \Phi_k$ is also independent of fillings. Since for the standard filling $V\times \D$, $\oplus \Phi_k$ is injective on the check component (the quotient of hat component). This implies that $\Phi$ must be injective for any filling on the check component. On the other hand, note that $1\in \Ima \Phi_0$ for $V\times \D$, therefore $1+A\in \Ima \Phi$ for some $A\in \oplus_{i>0}H^{2i}(Y)$ by the $\Z/2$ grading. By Proposition \ref{prop:kill}, we have that $\delta:SH_+^{*,\le 1+\delta,\le\epsilon}(W)\to H^{*+1}(W)$ is surjective and $SH^*(W)=0$.  Moreover, by Proposition \ref{prop:checkhat}, $\delta$ factors through the projection to the check component. Then the injectivity of $\Phi$ on the check component implies that $\delta$ is an isomorphism on the check component. Hence $H^*(W)\to H^*(Y)$ is also injective. Therefore, to finish the proof, it is sufficient to show that the image of $H^*(W)\to H^*(Y)$ is also independent of filling.  The injectivity of $\oplus \Phi_k$ implies that  $\Ima \oplus \Phi_k=\oplus \left((\Ima \Phi\cap F_kH^*(Y))/(\Ima \Phi\cap F_{k+1}H^*(Y))\right)$. On the other hand, the filtration on $H^*(Y)$ is the natural filtration by grading and $\Ima \Phi$ is the image of  $H^*(W)\to H^*(Y)$, where the associated graded groups are naturally isomorphic to the original groups. Hence $\Ima \oplus \Phi_k$ is naturally isomorphic to $\Ima \Phi$. The invariance of the former implies that $\Ima \Phi$ is independent of the filling, the claim follows. The claim on homology cobordism is from Proposition \ref{prop:homologycob} below.

When $W$ is only symplectically spherical, the symplectic action is well-defined for contractible orbits but not necessarily in the form on \eqref{eqn:action}. But since all the relevant orbits $\overline{\gamma}_p$ are contractible inside the cylindrical end of the boundary, the symplectic actions of those orbits are indeed given by \eqref{eqn:action}.  Therefore the same argument above goes through for symplectically spherical fillings. Note that $V\times \D$ is built from handles with indices at most $2n-1$. As a consequence, we have that $H^1(V\times \D) \to H^1(Y)$ is an isomorphism. Combining with the fact that $H^2(V\times \D)\to H^2(Y)$ is injective, we know that the symplectic form $\omega$ on a symplectically aspherical filling $W$ is necessarily exact and has a primitive whose restriction on the boundary is the original contact form. This proves that $W$ is an exact filling. 
\end{proof}

\subsection{The Calabi-Yau case}
First of all, the symplectic cohomology and positive symplectic cohomology are defined for Calabi-Yau fillings using the Novikov coefficient $\Lambda$ over $\Q$, see \cite[\S 8]{zhou2019symplectic}. In particular, the reason that positive symplectic cohomology is defined is no longer for action restrictions but because of the asymptotic behavior lemma \cite[Lemma 2.3]{cieliebak2018symplectic}. Similar to Proposition \ref{prop:kill}, we have the following analogue for Calabi-Yau fillings (i.e.\ a strong filling $W$ such that $c_1(W)$ is torsion) due to the fact that $1$ is a unit in $QH^*(W;\Lambda)$, which is $H^*(W;\Lambda)$ as a group.
\begin{proposition}\label{prop:kill2}
	Let $W$ be a Calabi-Yau filling and if $1$ is in the image of $SH^{*,\le a, \le b}_+(W;\Lambda) \to QH^{*+1}(W) $ is zero. Then $SH^*(W;\Lambda)=0$ and $SH^{*,\le a, \le b}_+(W;\Lambda)\to QH^{*+1}(W;\Lambda)$ is surjective.
\end{proposition}
\begin{proposition}\label{prop:exact}
	Let $W$ be a Calabi-Yau filling of $Y$, then $W$ is symplectically aspherical.
\end{proposition}
\begin{proof}
	In the Calabi-Yau case, we have a $\Z$ grading on symplectic cohomology. In particular,  we do not need Proposition \ref{prop:homology} to control the homology class and all generators of $C^*(K_{1+\delta}\oplus \epsilon(f-1))$ have a well-defined grading, since all of them are contractible. More explicitly, the grading is given by $|p\otimes \check{\gamma}_0|=\ind(p)-1, |p\otimes \hat{\gamma}_0|=\ind(p)-2$ and the Conley-Zehnder index of $\gamma_p$ is given by $n+2-\ind(p)$, in particular, the SFT degree $\mu_{CZ}(\gamma_p)+(n+1)-3$ is positive. There is also no need to use $H_{1+\epsilon+2\delta}$ or $\widetilde{H}_{1+\epsilon+2\delta}$. We can apply the same argument in Proposition \ref{prop:stretched} to $K_{1+\delta}\oplus \epsilon(f-1)$ directly, which shows that $\delta_{\partial}:H^*(C_+(K_{1+\delta}\oplus\epsilon(f-1)))\to H^*(Y;\Lambda)$ is independent of the Calabi-Yau filling. Then by Proposition \ref{prop:kill2}, we have $H^*(W;\Lambda)\to H^*(Y;\Lambda)$ is isomorphic to $H^*(V\times \D;\Lambda) \to H^*(Y;\Lambda)$ similar to the symplectically aspherical case. In particular, we have that $H^*(W;\Q)\to H^*(Y;\Q)$ is injective. As a consequence, $\omega$ is an exact form, i.e.\ $W$ is symplectically aspherical. 
\end{proof}
\begin{proof}[Proof of Theorem \ref{thm:main} for the Calabi-Yau case]
	It follows from Proposition \ref{prop:exact} and the symplectically aspherical case of Theorem \ref{thm:main}. 
\end{proof}

\begin{remark}
	When $c_1(V)=0$, it was shown in \cite{zhou2019symplectic} that $\partial (V\times \D)$ is asymptotically dynamically convex. For Calabi-Yau fillings of $\partial(V\times \D)$,  the index neck-stretching argument in \cite{zhou2019symplectic} requires that $\partial V$ has a Reeb dynamics with Conley-Zehnder indices bounded from below. But the index neck-stretching argument is still applicable if we attach additional flexible handles, while the action neck-stretching in this paper breaks down.
\end{remark}

\subsection{Homology cobordism}
A cobordism $W$ from $\partial_0W$ to $\partial_1W$ is called a homology cobordism iff $\partial_0W \to W$ and $\partial_1 W \to W$ both induce isomorphism on homology.
\begin{proposition}\label{prop:homologycob}
	Under the assumption in Theorem \ref{thm:main}, the filling $W$ can be obtained from $V\times \D$ by attaching a homology cobordism from $Y$ to $Y$.
\end{proposition}
\begin{proof}
	Let $W_0$ be a copy of $V\times \D(\epsilon) \subset V\times \D$ placed near $Y=\partial (V\times \D)$ for $\epsilon$ small. Then for any symplectically aspherical/Calabi-Yau filling $W$, we can assume $W_0$ is also contained in $W$. Let $X$ denote the cobordism from $\partial W_0$ to $\partial W$. We can assume $W_0$ is inside the strip $Y \times (1-3\epsilon,1)$ near the boundary. In particular, $H^*(Y \times (1-3\epsilon,1)) \to H^*(W_0)$ is an isomorphism when restricted to the image of $H^*(V\times \D )\hookrightarrow H^*(Y)$. Since $H^*(W)\to H^*(Y)$ is independent of filling, we have that $H^*(W)\to H^*(W_0)$ is an isomorphism. Therefore $H^*(X,\partial W_0)=0$ by excision. By Lefschetz duality and the universal coefficient theorem, we have that $H_*(X,\partial W_0)$ and $H_*(X,\partial W)$ are both zero. Hence $X$ is a homology cobordism.
\end{proof}

\subsection{General strong fillings}
The obstruction of applying Proposition \ref{prop:kill} for general strong fillings is that we may have a zero divisor $1+A$ in $QH^*(W;\Lambda)$ for $A\in \oplus_{i>0}H^{2i}(W;\Lambda)$, since $(1+A)\cup \cdot :QH^*(W;\Lambda)\to QH^*(W;\Lambda)$ is a linear map between finite dimensional $\Lambda$-spaces. For general strong fillings, symplectic cohomology and positive symplectic cohomology can be defined as usual if one applies a suitable virtual technique to overcome the transversality issue. For simplicity, we assume the strong filling is semi-positive as in \cite[Definition 6.4.1]{mcduff2012j}, so that the theory can be defined using generic almost complex structures. 

\begin{proposition}\label{prop:sphere}
	Assume $W^{2n+2}$ is a strong (semi-positive) filling, such that there is no embedded symplectic sphere $S$, with $2-n\le c_1(S) \le 2n-1$, then there is no zero divisor of $QH^*(W;\Lambda)$ in the form of $1+A$ for $A\in \oplus_{i>0}H^{2i}(W;\Lambda)$.
\end{proposition}
\begin{proof}
	If there is a such zero divisor, we claim there exists $B\in \oplus_{i>0} H^{2i}(W;\Lambda)$ such that $(1+A)\cup B=0$. First of all, there is a $\Z/2$ grading, hence there exist $B\in \oplus_{i>0} H^{2i}(W;\Lambda)$ and $b\in \Q$ such that $(1+A)\cup (b+B)=0$.  Note that $\langle C\cup D, 1\rangle=0$ for any $C,D\in \oplus_{i>0} H^{2i}(W;\Lambda)$, since the corresponding moduli space counts curves with a point constraint. However, such a moduli space must be empty as we can choose the point constraint near the boundary, and the maximal principle will obstruct such curve. Since $0=(1+A)\cup (b+B)=b+bA+B+A\cup B$, we must have $b=0$, i.e.\ the claim holds.
	Then to have $(1+A)\cup B=0$, the quantum product $\oplus_{i>0} H^{2i}(W;\Lambda)\otimes \oplus_{i>0} H^{2i}(W;\Lambda) \to \oplus_{i>0} H^{2i}(W;\Lambda)$ must be deformed, hence there must be some holomorphic sphere (possibly nodal) $S$, such that $6\le 2c_1(S)+2n+2 \le 6n$, i.e.\ $2-n\le c_1(S)\le 2n-1$. When $n\ge 2$, a (nodal) holomorphic sphere can be perturbed into an embedded symplectic sphere with the same first Chern class, hence a contradiction. When $n=1$, it is always semi-positive, and the curve contributing to the deformation of the product is necessarily somewhere injective by \cite[\S 6.6]{mcduff2012j}, hence the curve can be assumed to an embedded symplectic sphere, which is a contradiction. 
\end{proof}

\begin{proof}[Proof of Corollary \ref{cor:strong}]
	The proof follows from the same argument for Theorem \ref{thm:main}. Although we do not have a well-defined symplectic action for strong fillings, but the continuation maps used in the proof of Theorem \ref{thm:main} can be described by moduli spaces contained outside the boundary by neck-stretching, where the symplectic action is well-defined and can be used to restrict Floer trajectories. Then by Proposition \ref{prop:sphere}, we still have $SH^*(W)=0$ and $H^*(W;\Lambda)\to H^*(Y;\Lambda)$ is always injective. Hence the symplectic form is exact.
\end{proof}

\section{$h$-cobordisms}
In this section, we will upgrade the homology cobordism $X$ in Proposition \ref{prop:homologycob} to an $h$-cobordism assuming $\pi_1(Y)$ is abelian. Unlike the cohomology information, we can not quite get the full information on the fundamental group or more generally higher homotopy groups. But in the case of $\pi_1(Y)$ abelian, we do have enough ingredients to get some information on $\pi_1$ and conclude an $h$-cobordism.  
\subsection{Symplectic cohomology of covering spaces}

Recall from \cite[\S 3.3]{zhou2019symplectic}, for every covering space $\widetilde{W}\to W$, we can define the symplectic cohomology of the covering space. A cochain is a sum of formal sums of different lifts of periodic orbits on $W$. The differential is defined by lifting the differential on $W$ according to the (unique) parallel transportation. In particular, we have the following commutative long exact sequences,
\begin{equation}\label{eqn:dia}
\xymatrix{
\ldots \ar[r] & H^*(W) \ar[r]\ar[d] & SH^*(W) \ar[r]\ar[d] & SH^*_+(W)\ar[r]\ar[d] & H^{*+1}(W) \ar[r]\ar[d] & \ldots\\
\ldots \ar[r] & H^*(\widetilde{W}) \ar[r] & SH^*(\widetilde{W}) \ar[r] & SH^*_+(\widetilde{W})\ar[r] & H^{*+1}(\widetilde{W}) \ar[r] & \ldots
}
\end{equation}
Similarly for the filtered version. Note that we use only contractible orbits to define $SH^*(W)$. A map from the thrice punctured sphere (i.e.\ a pair of pants) can be completed (as a topological map) to a map from sphere, as all asymptotics are contractible orbits. As a consequence, we can lift the map to the universal cover and gives $SH^*(\widetilde{W})$ a unital ring structure. Then the first square in \eqref{eqn:dia} is a commutative square of unital rings.

\begin{remark}
	One can also define a symplectic cohomology with local coefficients, i.e.\ the underlying cochain complex is the free $\Z[\pi_1]$ module generated by periodic orbits. The differential is again the lifting of the ordinary differential and it respects the $\Z[\pi_1]$ module structure. The corresponding cohomology is the same as the symplectic cohomology of the universal cover if $\pi_1$ is finite. If $\pi_1$ is infinite, the symplectic cohomology of the universal cover allows generators which can be viewed as an infinite sum in the group ring. On the regular cohomology level, i.e.\ the Morse theory level, the cohomology with local coefficient $H^*(W;\Z[\pi_1])$ is the compactly supported cohomology of the universal cover $H_c^*(\widetilde{W})\ne H^*(\widetilde{W})$, see \cite[Proposition 3H.5]{hatcher2002algebraic}. It still carries a product structure (the pair of pants construction holds), but it is not unital ($1\in H^*(\widetilde{W})$ is represented by an infinite sum in the group ring).
\end{remark}

\begin{proposition}\label{prop:fundamental}
		If $\pi_1(Y)$ is abelian, then we have $\pi_1(V)$ is abelian  and $\pi_1(Y)\to \pi_1(V\times \D)$ is an isomorphism. Then for any exact/symplectically aspherical  filling $W$, we have $\pi_1(Y) \to \pi_1(W)$ is an isomorphism. Moreover, $H^*(\widetilde{W}) \to H^*(\widetilde{Y})$ is injective and independent of fillings for the universal covers.
\end{proposition}
\begin{proof}
	Since $V\times \D$ can be built from handles with indices at most $2n-1$ (i.e.\ co-indices at least $3$), we have $\pi_1(Y)\to \pi_1(V\times \D)$ is an isomorphism. In particular, $H_1(Y)\to H_1(V\times \D)$ is an isomorphism. Then by universal coefficient theorem, we have $H^1(V\times \D) \to H^1(Y)$ and $\tor H^2(V\times \D) \to \tor H^2(Y)$ are isomorphisms.  As a consequence, \eqref{t1} of Theorem \ref{thm:main} implies that $H_1(Y)\to H_1(W)$ is an isomorphism. Therefore $\pi_1(Y)\to \pi_1(W)$ is at least injective.  Then $\pi_1(Y)\to \pi_1(W)$ is surjective by the same argument in \cite[Theorem 3.16]{zhou2019symplectic} by considering the symplectic cohomology of the universal cover. The only difference is replacing the grading in \cite{zhou2019symplectic}  with  the associated graded group from the filtration. Hence $\pi_1(Y)\to \pi_1(W)$ is an isomorphism. The independence of $H^*(\widetilde{W})\to H^*(\widetilde{Y})$ then follows from the same proof of Theorem \ref{thm:main}.
\end{proof}

\begin{lemma}\label{lemma:ext}
	For a $\Z$-module $A$, if $\Hom(A,\Z)=\Ext(A,\Z)=0$, then $A=0$.\footnote{This is from Eric Wofsey's solution to \href{https://math.stackexchange.com/questions/1734222/does-trivial-cohomology-imply-trivial-homology-does-operatornamehoma-math}{https://math.stackexchange.com/questions/1734222/does-trivial-cohomology-imply-trivial-homology-does-operatornamehoma-math}}
\end{lemma}
\begin{proof}
	Since $\Ext(\cdot,\Z)$ turns injective maps into surjective maps by $\Ext^2(\cdot,\Z)=0$, we have $\Ext(B,\Z)=0$ for any $B\subset A$. Therefore any finitely generated subgroup of $A$ is free. Hence $A$ is torsion free. Next we fix a prime $p$. Since $A$ is torsion free, we have short exact sequence $0\to A \stackrel{p\times }{\to} A \to A/pA \to 0$, which induces exact sequence
	$$\Hom(A,\Z)\to \Ext(A/pA, \Z) \to \Ext(A,\Z).$$
	Hence $\Ext(A/pA,\Z)=0$. But $A/pA$ is a direct sum of copies of $\Z/p$. Therefore $A/pA=0$ for any $p$. Then $A$ is a divisible torsion free group, hence a $\Q$-vector space. Since $\Ext(\Q,\Z)\ne 0$, we have $A=0$.
\end{proof}

We also need the following form of the universal coefficient theorem.
\begin{lemma}\label{lemma:univ}
Let $R$ be a ring and  $(C_*,\partial)$ be a cochain complex of $R$ modules such that $H_*(C_*)=0$, then $H^*(\hom_R(C_*,R))=0$. 
\end{lemma}
\begin{proof}
    We use $B_n\subset C_n$ to denote the image of $\partial$ and $Z_n\subset C_n$ to denote the kernel of $\partial$. By assumption, we have $B_n=Z_n$. Note that we have  a tautological short exact sequence,
    $$0\to Z_n\to C_n\to B_{n-1}\to 0.$$
    Then the tautological short exact sequence 
    $$0\to Z_n/B_n \to C_n/B_n \to C_n/Z_n \to 0$$
    is 
    $$0\to 0 \to C_n/B_n \to B_{n-1}\to 0.$$
    If we use $M^*$ to denote $\hom_R(M,R)$ for a $R$-module $M$, then we have the following
    $$
    \xymatrix{
    & & 0 & & \\
    B_n^* & C_n^* \ar[l] & \left(\displaystyle\frac{C_n}{B_n}\right)^* \ar[l]\ar[u] & 0\ar[l] & 0\\
    & & B_{n-1}^*\ar[u] & & Z_{n-1}^*\ar[ll]^{\simeq} \ar[u]  \\
    & & 0\ar[u] & & C_{n-1}^*\ar[ull]\ar[u]}
    $$
    where all vertical and horizontal lines are exact.
    
    Since the coboundary $\partial^*$ on $C^*$ is defined as the map from $C_{n-1}^*$ to $C_n^*$ in the diagram above. As a consequence, we have $\ker \partial^*$ is $\ker [C_{n-1}^*\to B_{n-1}^*]$. On the other hand, we have $\Ima \partial^*=\Ima (Z_{n-2}^*\to C^*_{n-1})$. Since $B_{n-2}^*\simeq Z_{n-2}^*$ and $(C_{n-1}/B_{n-1})^*\simeq B_{n-2}^*$, we have $\Ima \partial^*=\Ima ((C_{n-1}/B_{n-1})^*\to C^*_{n-1})=\ker [C_{n-1}^*\to B_{n-1}^*] = \ker \partial^*$ by the exactness of the topmost row. That is $H^*(\hom_R(C_*,R))=0$.
\end{proof}
\begin{remark}
    The general universal coefficient theorem (which computes cohomology from homology) needs some assumptions on $R$ or $C_*$, see \cite[\S 3.6.5]{MR1269324}. The above version without any assumption works because $C_*$ is acyclic, which is sufficient for our purpose in Proposition \ref{prop:cob}.
\end{remark}

\begin{proposition}\label{prop:cob}
	If $\pi_1(Y)$ is abelian, then any exact/symplectically aspherical filling $W$ is $V\times \D$ glued with an $h$-cobordism from $Y$ to $Y$.
\end{proposition}
\begin{proof}
	Let $X$ be the homology cobordism from $\partial W_0$ to $\partial W$ in the proof of Proposition \ref{prop:homologycob}. By Proposition \ref{prop:fundamental}, we know that $\partial W_0\hookrightarrow W_0\hookrightarrow W$ both induce isomorphisms on $\pi_1$. Then the van Kampen theorem implies that the following push-out diagram consists of isomorphisms,
	$$\xymatrix{\pi_1(\partial W_0) \ar[r]\ar[d] & \pi_1(W_0)\ar[d]\\
		\pi_1(X) \ar[r] & \pi_1(W)}$$
	Then we have $\pi_1(\partial W) \to \pi_1(X)$ is an isomorphism since  $\pi_1(\partial W) \to \pi_1(W)$ is an isomorphism. Applying the argument in Theorem \ref{thm:main} to the universal cover, we have $H^*(\widetilde{W})\to H^*(\widetilde{W_0})$ is an isomorphism by Proposition \ref{prop:fundamental}. Hence by excision, we have $H^*(\widetilde{X}, \widetilde{\partial W_0})=0$. Then by universal coefficient and Lemma \ref{lemma:ext}, we have $H_*(\widetilde{X},\widetilde{\partial W_0})=H_*(X,\partial W_0;\Z[\pi_1])=0$.
	Then by Lemma \ref{lemma:univ} for $R=\Z[\pi_1]$, we have $H^*(X,\partial W_0;\Z[\pi_1])=0$. Then by the Lefschetz duality with twisted coefficients, $H_*(X,\partial W;\Z[\pi_1])=H_*(\widetilde{X},\widetilde{\partial W})=0$. Therefore $X$ is an $h$-cobordism by Whitehead's theorem. 
\end{proof}

\begin{proof}[Proof of Theorem \ref{thm:diff}]
	If the Whitehead group of $Y$ is trivial, then the $h$-cobordism is a trivial cobordism by the $s$-cobordism theorem \cite{milnor1966whitehead}. Hence $W$ is diffeomorphic to $V\times \D$. In general, we can apply the Mazur trick, see \cite{milnor1966whitehead}. That is there is an $h$-cobordism $X'$ from $Y$ to itself such that the concatenations $X\circ X'$ and $X'\circ X$ are trivial cobordisms. Note that $\mathring{W}$ is diffeomorphic to $\widehat{W}$, which is diffeomorphic to $\ldots \circ X\circ X'\circ W$, i.e.\ attaching infinite $X\circ X'$ to $W$. On the other hand, it is $\ldots \circ X \circ X'\circ X \circ V\times \D$, which is diffeomorphic to $\widehat{V}\times \C$ or the interior of $V\times \D$.    
\end{proof}
\begin{remark}
	The Whitehead torsion can be put into the framework of Floer theories \cite{abouzaid2018simple}. One can prove that the Whitehead torsion of the cochain map underlying the isomorphism $SH_+^*(W) \to H^{*+1}(Y)\to H^{*+1}(V\times \{1\})$ has zero Whitehead torsion assuming $\pi_1(Y)$ is abelian. What we still need is that the Whitehead torsion of $SH_+^*(W) \to H^{*+1}(W)$ is zero. 
\end{remark}

It is very likely that the diffeomorphism type of the filling is unique for any Liouville domain $V$, or at least the homotopy type is unique. However, this requires better ways to probe homotopy groups of the filling, hence we end the paper by asking the following question.
\begin{question}
    Is there a Floer theoretic interpretation of homotopy groups? In particular, is it true that $\pi_k(W)\to \pi_k(Y)$ are independent of exact fillings for any Liouville domain $V$ and $k\ge 1$? What can we say about the Whitehead product on $\pi_k(W)$.	
\end{question}



\bibliographystyle{plain} 
\bibliography{ref}

\end{document}